\documentclass{siamltex}
\usepackage{multirow}
\usepackage{graphicx}
\usepackage{mathrsfs}
\usepackage{color}
\usepackage{threeparttable,booktabs}
\usepackage{algorithm,algorithmic}
\usepackage{amsmath,amsfonts,amssymb,graphicx}%
\usepackage{subfigure}
\usepackage{caption}
\usepackage{rotating}
\usepackage{lscape}
\usepackage{makeidx}
\usepackage{cite}
\usepackage{hyperref}
\hypersetup{colorlinks}
\usepackage{colortbl}

\newtheorem{remark}[theorem]{Remark}

\def \td{\tilde}

\def\orth{{\rm orth}}

\def\eig{{\rm eig}}

\def\span{{\rm span}}
\def\det{{\rm det}}
\def\WA{\mathscr{W}}

\def\Krylov{{\rm Krylov}}

\def\RK{{\rm RK}}
\def\cks{{\rm CK}}
\def\CK{{\rm CK}}
\def\ie{{\rm i.e.}}
\def\st{{\rm s.t.}}

\def\ARKSM{\textsf{ARKSM}}
\def\AAA{\textsf{AAA}}
\def\IRKA{\textsf{IRKA}}

\def\twosided{\textsf{two-sided}}

\title{transfer function interpolation remainder formula \\ of rational Krylov subspace methods%
 \thanks{Version of Dec 27, 2024.}}

\author{Yiding Lin\thanks{School of Economic and Mathematics, Southwestern University of Finance and Economics,  Chengdu, China ({\tt yiding.lin@gmail.com}). The work is supported by  National Natural Science Foundation of China (NSFC-11526166, NSFC-12101508).}}

\begin{document}
\bibliographystyle{siam}
\maketitle

\begin{abstract}
Rational Krylov subspace projection methods have proven to be a highly successful approach in the field of model order reduction (MOR), primarily due to the fact that  some derivatives of the approximate and original transfer functions are identical.
This is the well-known theory of moments matching. 
Nevertheless, the properties of points situated at considerable distances from the interplation nodes remain underexplored.
In this  paper, 
we present the explicit expression of the MOR error, which involves both the shifts and the Ritz values.
The superiority of our discoveries over the known moments matching theory can
be likened to the disparity between the Lagrange and Peano type remainder formulas  in Taylor's theorem.
Furthermore, two explanations are provided for the error formula with respect to the two parameters in the resolvent function.
One explanation reveals that the  MOR error is an interplation remainder, while  the other  explanation  implies that the error is also a  Gauss-Christoffel quadrature remainder.
By applying the error formula, we suggest a greedy algorithm for  the interpolatory  $H_{\infty}$ norm MOR.

%

%

\end{abstract}



\pagestyle{myheadings}
\thispagestyle{plain}



\markboth{Yiding Lin}{ Transfer function interpolation remainder formula of rational Krylov subspace methods}


\section{Introduction}
We  investigate  the model error of model order reduction (MOR)\cite{MR2516399}, which involves  rational Krylov subspace projection methods.
The initial single-input-single-output (SISO) dynamical system is  described as follows.
\begin{equation*}
\left\{
\begin{aligned}
\frac{dx(t)}{dt}&=Ax(t)+bu(t),\\
y(t)&=c^Hx(t),\\
\end{aligned}
\right.
\end{equation*}
with
$A \in \mathbb{C}^{n \times n}$ and $b,c\in \mathbb{C}^{n \times 1}$.

Using  the rational Krylov subspaces $\span(V)$ and $\span(W)$, we  obtain  the reduced system:
\begin{equation*}
\left\{
\begin{aligned}
W^HV\frac{dx(t)}{dt}&=W^HAVx(t)+W^Hbu(t),\\
y(t)&=c^HVx(t).\\
\end{aligned}
\right.
\end{equation*}
The transfer functions of these systems can be described as follows.
\begin{equation*}
\begin{aligned}
h(\mathfrak{s})&=c^H(\mathfrak{s}I-A)^{-1}b,\\
\tilde{h}(\mathfrak{s})&=c^HV(\mathfrak{s}W^HV-W^HAV)^{-1}W^Hb.\\
\end{aligned}
\end{equation*}

The model error can be measured by utilizing specific norms of
$e(\mathfrak{s}):=h(\mathfrak{s})-\tilde{h}(\mathfrak{s})$ \cite[Section  7.2.3]{MR3672144}.
The well-known moments matching result states that  $h^{(k)}(s_0)=\tilde{h}^{(k)}(s_0).$ This implies
$e(\mathfrak{s})=o((\mathfrak{s}-s_0)^{k})$ near $s_0$. This type of error can be likened to the Peano remainder in Taylor's theorem.
When we do analysis on a wide region, this Peano type remainder is not satisfied.
We are curious about the behavior of $e(\mathfrak{s})$ as $\mathfrak{s}$ moves away from $s_0$.
%
In a word, this paper aims to address the following question: Considering that $\tilde h(\mathfrak{s})$ is the interpolating function of $h(\mathfrak{s})$, who can be identified as  the interpolation remainder?
The complexity of this problem stems from  the fact that both $h(\mathfrak{s})$ and $\tilde h(\mathfrak{s})$  are rational  functions and possess special forms.
As a result, the existence of an explicit error formula was uncertain until the completion of this research.
The Krylov subspace projection method is  one of the mainstream methods  in  the field of  MOR.
There are plenty of references related to this topic. We mainly review the literature  which are closely related to the (tangential) interpolation property and  the moments matching theory.
This type methods are first set up by Skelton \textit{et al}. \cite{MR924278,MR794690,MR780321}.
Then, Grimme  
combines it with the rational Krylov subspaces \cite{grimme1997krylov}.
Gallivan, Vandendorpe and Van Dooren  make several contributions to this topic \cite{MR2013459,MR2043508,MR2124150}.
A recent book \cite{doi:10.1137/1.9781611976083} by 
Antoulas, Beattie and Gugercin summarizes the latest  theories and algorithms.
Other material can be found in \cite{MR2421462,MR2155615,MR2142591,MR3274477,MR3419868,MR2164197,MR2516399} and references therein.

In 1997, Grimme  \cite{grimme1997krylov}  observed that the model error can be expressed as
$e(\mathfrak{s})=r_c^H(\mathfrak{s}I-A)^{-1}r_b.$ This truth is used in later references (e.g.\cite{MR3473700,MR3592472,MR2501562,frangos2007adaptive,panzer2014model}).
Our work starts by expanding $r_b$ and $r_c$ in relation to the Ritz values and the shifts.
The expressions for  the residual  $r$ via rational Krylov subspace methods are presented in
\cite{MR2854603,MR2858340,guttel2010rational,MR2485440,MR3570279}.
The known expressions of  the residual concern  the Galerkin type (one-sided) projection process. 
When we generalize the residual expression for dealing with MOR, we replace it by the
Petrov-Galerkin (two-sided) projection process, which is widely used in MOR\cite{1586750,panzer2014model}.
We explicitly express $r_b$ and $r_c$ through the utilization of special bases derived from the Lanczos biorthogonalization procedure.
Thus,   an explicit  formula  of  $e(\mathfrak{s})$ is drived.

In the preceding research, we identify two further theoretical explanations for the model error. 
The explanations pertain to   the two parameters in the resolvent function $\mathcal {H}(\lambda,\mathfrak{s})=1/(\mathfrak{s}-\lambda)$.
 With respect to the variable $\mathfrak{s}$, it can be observed that the error formula is, in fact, the interpolation remainder when the Hermitian formula is applied to the resolvent function.

 The situation is more complicated with respect to the variable $\lambda$. It should be noted the Ritz values  correspond to  the quadrature nodes  of  Gauss-Christoffel quadrature. This motivates our research into the discovery of a  correlation between the error formula and Gauss quadrature.
The relevant theory for Hermitian $A$ and $b=c$ is presented in  book \cite[Chapter 3]{MR3024841} by Liesen and Strako{\v{s}}.  
After a series of derivations, it can be concluded that the error formula is, in fact, a Gauss-Christoffel quadrature remainder. 

Based on various error estimations, a multitude of adaptive algorithms for MOR have been developed \cite{MR2501562,MR3592472,panzer2014model}.
In designing algorithms for the interpolatory $H_{\infty}$ norm MOR, it is not appropriate to completely compute $e(s)$, since the explicit form of $e(s)$ contains the term $(\mathfrak{s}I-A)^{-1}$ inside.
Following a thorough examination of the explicit error formula, we are able to derive certain approximations that can be computed in the reduced problems.
Our two-sided greedy algorithms are then presented by modifying the one-sided projection algorithm in \cite{MR2858340}.
The next  shifts are obtained by identifying the maximal point of the current error. 
Subsequently, the proposed algorithms are  evaluated in comparison with other algorithms.
The numerical experiments demonstrate a comparable behaviour of its error in the $H_\infty$ norm to that of  \IRKA~(which computes $H_2$ norm optimal MOR),  while our algorithms require a considerably lesser CPU time.



The paper is structured as follows:
In Section \ref{sect:main_result_section}, we present the explicit  error formula for  the  two-sided projection method, together with a proof and  two further explanations.
Section \ref{sect:one_sided_error_formula_section} provides the explicit  error formula for  the  one-sided projection method.
A similarity can be observed between the error in the $l$-order two-sided projection MOR and that of the $2l$-order one-sided projection MOR.
In Section \ref{sect:H_infty_norm_section}, we present a review of the interpolatory $H_\infty$ norm MOR and provide approximations of $e(\mathfrak{s})$ which can be computed in the reduced problems.
Our  greedy two-sided algorithms are presented in Section \ref{section:greedy_algorithms_two_sided}, and the numerical testing on benchmark problems is provided in Section \ref{sect:numerical_experiments}.
The conclusions are outlined in Section \ref{sect:conclusion_section}.

%

\noindent {\bf Notation}:
%
The standard Krylov subspace is represented by the following symbol:
\begin{equation}\label{eqn:standard_Krylov_definition}
	\begin{aligned}
		\Krylov(A,b,l):=\span(b,Ab,\dots,A^{l-1}b).
	\end{aligned}
\end{equation}
The operation $\orth(V)$ produces the orthonormal basis matrix of $\span{(V)}$. 
The shifts sets for the left and right rational Krylov  subspaces are, respectively, denoted by 
$\mathbb{T}=\{t_j\}_{j=1}^{k_c}$
and $\mathbb{S}=\{s_j\}_{j=1}^{k_b}$. 
Consequently, the rational Krylov subspaces can be written as follows:
\begin{equation}\label{eqn:pure_RK_definition}
	\begin{aligned}
		\RK(A,b,\mathbb{S},k_b)&:=\span\{(A-s_1I)^{-1}b,(A-s_2I)^{-1}(A-s_1I)^{-1}b ,\ldots ,\prod\limits_{j=1}^{k_b}(A-s_jI)^{-1}b\},\\
\RK(A^H,c,\overline{\mathbb{T}},k_c)&:=\span\{(A-t_1I)^{-H}c,(A-t_2I)^{-H}(A-t_1I)^{-H}c ,\ldots ,\prod\limits_{j=1}^{k_c}(A-t_jI)^{-H}c\}.\\
	\end{aligned}
\end{equation}


The symbol $\mathbb{P}_m(\lambda)$ denotes the polynomial set, in which the degree of the polynomials is limited to $m$ or less.
In accordance with the polynomial  $\varphi(\lambda),$ the symbol $\mathbb{P}_{m}(\lambda)/\varphi(\lambda)$  represent  the rational function set.
The numerator of its element is a polynomial  of degree at most $m$.
In addition, the symbol 
$\mathbb{Q}_{l-1,l}(\lambda)$ denotes the rational function set. The numerator of its element  is of   degree $l-1$, while the denominator is of  degree $l$.


Write $\iota=\sqrt{-1}.$ 
The set of $A$'s eigenvalues is denoted by $\eig(A)=\{\lambda_i(A)\}_{i=1}^n$.
The symbol $f[x_1,x_2, \ldots, x_m]$ denotes the divided differences  of the function $f(x)$.
Unless otherwise specified, the norm $\| \cdot\|$ is an abbreviation of 
$\| \cdot\|_2$.
The symbol $\Sigma^-$ denotes the  complement set of $\Sigma$.
Matlab notations will be employed wherever feasible.

\section{The error formula  of the two-sided projection method}\label{sect:main_result_section}
We commence with the definition of  the  combined Krylov (CK) subspaces.
%
%
%
%
\begin{equation}\label{eqn:cks_basis_b}
	\begin{aligned}
\cks(A,b,k_b,m_b)&:=\RK(A,b,\mathbb{S},k_b)+ \Krylov(A,b,m_b), \\
		\cks(A^H,c,k_c,m_c)&:=\RK(A^H,c,\overline{\mathbb{T}},k_c)+ \Krylov(A^H,c,m_c), \\
	\end{aligned}
\end{equation}
where (rational) Krylov subspaces
are defined  by \eqref{eqn:standard_Krylov_definition} and \eqref{eqn:pure_RK_definition}.
With notations $l:=k_b+m_b=k_c+m_c$ and
\begin{equation}\label{eqn:prod_all_shifts}
\begin{aligned}
\varphi(\mathfrak{s}):=\prod\limits_{j=1}^{k_b}(\mathfrak{s}-s_j),
\qquad
\psi(\mathfrak{s}):=\prod\limits_{j=1}^{k_c}(\mathfrak{s}-t_j),\\
\end{aligned}
\end{equation}
we observe that (cf. Section \ref{sec:isomorphism_section})
\begin{equation*}
\begin{aligned}
\cks(A,b,k_b,m_b)=\Krylov(A,\varphi(A)^{-1}b,l), \qquad
\cks(A^H,c,k_c,m_c)=\Krylov(A^H,\psi(A)^{-H}c,l).\\
\end{aligned}
\end{equation*}

\begin{theorem}\label{thm:main_result_1}
Let $V$ and $W$ satisfy
$\span(V)=\cks(A,b,k_b,m_b),  \span(W)=\cks(A^H,c,k_c,m_c)$
defined by \eqref{eqn:cks_basis_b}.
Suppose that  the Lanczos biorthogonalization procedure of $[A,\varphi(A)^{-1}b,\psi(A)^{-H}c]$ does not break down until the $(l+1)$th iteration.
Let $\Lambda(\lambda):=\prod_{i=1}^{l}(\lambda-\lambda_i)$
be the monic characteristic polynomial of $(W^HV)^{-1}W^HAV$.
With \eqref{eqn:prod_all_shifts}, write
\begin{equation}\label{eqn:g_b_g_c_expression}
	\begin{aligned}
 g_b(\lambda):=\frac{\Lambda(\lambda)}{\varphi(\lambda)},\quad   g_c(\lambda):=\frac{\Lambda(\lambda)}{\psi(\lambda)}.
\end{aligned}
\end{equation}

Then,  it holds that
\begin{equation*}
\begin{aligned}
e(\mathfrak{s})=h(\mathfrak{s})-\tilde{h}(\mathfrak{s})&=c^H(\mathfrak{s}I-A)^{-1}b-c^HV(\mathfrak{s}W^HV-W^HAV)^{-1}W^Hb\\
&=\frac{1}{g_b(\mathfrak{s})g_c(\mathfrak{s})}c^Hg_c(A)(\mathfrak{s}I-A)^{-1}g_b(A)b.\\
\end{aligned}
\end{equation*}

\end{theorem}

The proof is presented in three {\it steps}:
Theorem \ref{thm:arnoldi_speical_basis_result},
Theorem \ref{thm:general_subspace_speical_basis} and the final proof 
in  Section \ref{subsection:main_result_proof}.
Furthermore, we possess  Theorem \ref{th:main_result_include_E}
for the generalization onto the descriptor system.  
The  proof is a constructive type one by using a Lanczos biorthogonalization procedure. The aforementioned assumption ensures $W^HV$ is nonsingular. 
We commence by presenting a series of observations.

\begin{enumerate}
\item The combined Krylov subspaces span the same subspaces as  those described in  \cite[Definition 11]{MR2013459}. 
The definition \eqref{eqn:cks_basis_b} employs the product-type bases, with no consideration on the multiplicity of the shifts.
These special  bases are a truncation of the one  used in \cite[Theorem 6.1]{MR1920565}.
They are one of the various bases of the rational Krylov subspace.
The expressions of rational functions make it evident that they span the same subspace (cf. \cite[Lemma 4.2(d)]{guttel2010rational}).


\item It should be noted that $k_b+m_b=k_c+m_c$, but that $k_b$ does not need to be equal to $k_c$.
In order to encompass the greatest number of significant subspaces, our rational  Krylov subspace was established as inclusively as possible.
Theorem~\ref{thm:main_result_1} presents a comprehensive result concerning the interpolation property. 
It directly gives rise to the moments matching  results (e.g.\cite[Proposition 11.7, Proposition 11.8, Proposition 11.10, Proposition 11.11]{MR2155615}). 
From  the  explicit form of $e(\mathfrak{s})$, it is apparent  that $e(\mathfrak{s})$ fully satisfies
\cite[Definition 16]{MR2013459}
(an equivalent condition of moments matching).


\item It is observed that the Krylov subspace type methods are incapable of maintaining the stability of the system. 
The reasons are now apparent from an examination of  Theorem \ref{thm:main_result_1}.  
The error formula is not reliant on the stability of the matrix $A$.  Specifically, the condition of  Theorem \ref{thm:main_result_1} does not require \eig$(A)$ be in the left  half plane. 
Therefore, it is not reasonable to anticipate that this two-sided projection method will maintain stability unless additional constraints are introduced.
Note that there exist some other applications which do not involve the stability of  $A$
(e.g. \cite[(3.4)]{MR2776701} \cite[Theorem 3.1]{MR4392235}).

%

\end{enumerate}


\subsection{{\it Step 1:} Special subspaces (standard Krylov subspaces) and special bases}
We first prove a special case of 
Theorem \ref{thm:main_result_1} when $l=m_b=m_c =:m, 
 k_b=k_c=0.$ The main result of this {\it step} is Theorem  \ref{thm:arnoldi_speical_basis_result}.
\subsubsection{Lanczos biorthogonalization procedure}\label{section:unsymmetric_Lanczos}


 Our proof utilizes the bases derived from the Lanczos biorthogonalization procedure of  $(A,b,c)$  (cf. \cite[Section 7.1]{MR1990645}).  
 Thus, it is necessary to assume that the procedure does not break down.
The conditions for the successful execution of the procedure have been well researched\cite{MR1201315,MR4311639}.
The relations resulting from the Lanczos biorthogonalization procedure of  $(A,b,c)$ are concluded by:
\begin{equation}\label{eqn:Krylov_lanczos}
\begin{aligned}
	v_1&=b,\quad w_1=c,\quad 
	c^Hb=1,\\
AV_m&=V_{m+1}\underline{T}_m, \quad \span(V_m)={\rm Krylov}(A,b,m),
\quad \span(V_{m+1})={\rm Krylov}(A,b,m+1),\\
A^HW_m&=W_{m+1}\underline{K}_m, \quad \span(W_m)={\rm Krylov}(A^H,c,m),
\quad \span(W_{m+1})={\rm Krylov}(A^H,c,m+1),\\
W_m^HV_m&=I_m, \quad W_{m+1}^HV_{m+1}=I_{m+1}, \quad
W_m^HAV_m=T_m,
\quad  V_m^HA^HW_m=K_m,
\quad T_m=K_m^H,\\
\underline{T}_m&=
\left[ \begin{array}{ccccc}
\alpha_1& \beta_2 & &&\\
\gamma_2 & \alpha_2 &\beta_3 &&\\
 &   \ddots & \ddots&\ddots&\\
&&\gamma_{m-1}&\alpha_{m-1}&\beta_m\\
&&&\gamma_m&\alpha_m\\
\hline
&&&&\gamma_{m+1}\\
\end{array}
\right]
=\left[ \begin{array}{c}
	T_m\\
	\hline
	\gamma_{m+1}e_m^H
	\end{array}\right].\\
\end{aligned}
\end{equation}

\begin{lemma}\label{lm:polynomial_expression_taub}
Let $\tau_j(\lambda)=a_j^j\lambda^j+a_{j-1}^j\lambda^{j-1}\cdots+a_1^j\lambda+a_0^j$ be a  polynomial  of degree $j$.
By \eqref{eqn:Krylov_lanczos}, we have 
\begin{eqnarray*}
	\tau_j(A)b&=&V_m\tau_j(T_m)e_1,      \qquad\qquad\qquad\qquad\  j<m\\
	\tau_m(A)b&=&V_m\tau_m(T_m)e_1+a_m^m\zeta_mv_{m+1},\qquad j=m
\end{eqnarray*}
where $\zeta_m=\gamma_{m+1}\cdots\gamma_3\gamma_2.$
\end{lemma}

\begin{lemma}\label{lm:Krylov_Cayley}
	Suppose relations \eqref{eqn:Krylov_lanczos} hold.
	Let $r$ be a nonzero vector
	in $\Krylov(A,b,{m+1})$  that is  orthogonal to
	$\Krylov(A^H,c,m)$.
	It follows that the vector $r$ can be expressed as  $r=\rho \Lambda_m(A)b,$ where  the scalar $\rho \neq 0$
	and $\Lambda_m(\lambda)= \det(\lambda I- T_m)$ is the  monic characteristic polynomial of $T_m$.	
	
\end{lemma}
The above lemmas are generalizations of
\cite[Lemma  3.1, Lemma  3.2]{MR1716118} or \cite[Lemma 2.1, Corollary 2.1]{MR1323816}.
As the proofs are quite similar, we do not provide
the details. 

\subsubsection{The derivation}\label{subsect:infinty_shifts_special_basis}


Similar to the case  of  the linear equations  \cite[Chapter 5]{MR1990645}, we require the residual be orthogonal to the given subspace.  With notations
\begin{equation}\label{eqn:r_b_r_c_definition}
\begin{aligned}
r_b:=b-(\mathfrak{s}I-A)x_b,
\quad r_c:=c-(\mathfrak{s}I-A)^Hx_c,\\
\end{aligned}
\end{equation}
we impose the Petrov-Galerkin conditions:
\begin{equation*}
\begin{aligned}
 x_b&\in{\rm Krylov}(A,b,m)=\span(V_m), \quad  &\st \quad r_b&\bot{\rm Krylov}(A^H,c,m)=\span(W_m),\\
%
%
%
{\rm and} \quad x_c&\in{\rm Krylov}(A^H,c,m)=\span(W_m), \quad  &\st \quad r_c&\bot{\rm Krylov}(A,b,m)=\span(V_m).\\
\end{aligned}
\end{equation*}
Thus,  we get $\tilde{h}(\mathfrak{s})=x_c^Hb=c^Hx_b,$ where
	\begin{equation*}
		\begin{aligned}
x_b=V(\mathfrak{s}I-W^HAV)^{-1}W^Hb,
\quad
x_c=W(\mathfrak{s}I-W^HAV)^{-H}V^Hc.\\
\end{aligned}
\end{equation*}

With \eqref{eqn:r_b_r_c_definition}, we easily observe that:
For any complex $\mathfrak{s}$, it holds that
\begin{equation*}
\begin{aligned}
r_b\in{\rm Krylov}(A,b,{m+1})=
\span(V_{m+1}),
\quad
r_c\in{\rm Krylov}(A^H,c,{m+1})
= \span(W_{m+1}).
\end{aligned}
\end{equation*}

We define the polynomial expressions:
\begin{equation*}
\begin{aligned}
x_b&=\chi_b(A,\mathfrak{s})b, \quad \chi_b(\lambda,\mathfrak{s}) \in \mathbb{P}_{m-1}(\lambda),\\
r_b&=\gamma_b(A,\mathfrak{s})b,\quad \gamma_b(\lambda,\mathfrak{s}) \in \mathbb{P}_{m}(\lambda).\\
\end{aligned}
\end{equation*}
Their relations are given by
\begin{equation*}
	\begin{aligned}
		r_b&=b-(\mathfrak{s}I-A)x_b,\\
		\gamma_b(\lambda,\mathfrak{s})&=1-(\mathfrak{s}-\lambda)\chi_b(\lambda,\mathfrak{s}),\\
		\chi_b(\lambda,\mathfrak{s})&=\frac{1-\gamma_b(\lambda,\mathfrak{s})}{\mathfrak{s}-\lambda}.\\
	\end{aligned}
\end{equation*}
We easily find $\gamma_b(\lambda,\mathfrak{s})$ satisfies
a constraint condition:
\begin{equation}\label{eqn:r_b_constraint_condition_infty}
	\begin{aligned}
\gamma_b(\mathfrak{s},\mathfrak{s})=1.
\end{aligned}
\end{equation}
%

Similar to \cite[(3.8)]{MR1323816} and \cite[Theorem 3.1]{MR1716118}, we deduce  the polynomial formulas for the residuals.
\begin{lemma}\label{lm:Krylov_rb_rc}
Suppose relations \eqref{eqn:Krylov_lanczos} hold.
Then, it holds that 
%
%
\begin{equation*}
\begin{aligned}
r_b&=\gamma_b(A,\mathfrak{s})b, \qquad r_c^H=c^H\gamma_c(A,\mathfrak{s}),\\
\gamma_b(\lambda,\mathfrak{s})& =\gamma_c(\lambda,\mathfrak{s})
=\frac{\Lambda(\lambda)}{\Lambda(\mathfrak{s})}, 
\end{aligned}
\end{equation*}
where $\Lambda(\lambda):=\prod_{i=1}^{m}(\lambda-\lambda_i)$
is the monic characteristic polynomial of $T_m=W^HAV$.
\end{lemma}

\begin{proof}
We observe that
$r_b\in{\rm Krylov}(A,b,{m+1})$
and $r_b \bot {\rm Krylov}(A^H,c,m).$
%
%
By   Lemma \ref{lm:Krylov_Cayley}, we directly obtain 
$r_b=\rho \Lambda(A)b,$
where $\Lambda(\lambda)$
is the monic characteristic polynomial of $T_m=W^HAV$.
Obviously, when $\mathfrak{s}$ varies, $\rho$ changes. Thus, we can regard $\rho$ as a function of  the variable $\mathfrak{s}$.
%
It implies that
$\gamma_b(\lambda,\mathfrak{s}) =\rho(\mathfrak{s})\Lambda(\lambda).$
By the constraint condition \eqref{eqn:r_b_constraint_condition_infty}, we directly obtain
$\rho(\mathfrak{s})=1/{\Lambda(\mathfrak{s})}$
and then
$\gamma_b(\lambda,\mathfrak{s})
=\Lambda(\lambda)/\Lambda(\mathfrak{s}).$
Similarly, we  get the result about $r_c$.
\end{proof}
%

\begin{theorem}\label{thm:arnoldi_speical_basis_result}
Suppose the Lanczos biorthogonalization procedure of $(A,b,c)$ does not break down until the $(m+1)$th iteration.
The relations of the formed matrices are given by (\ref{eqn:Krylov_lanczos}).
%
Then, it holds that
\begin{equation*}
\begin{aligned}
h(\mathfrak{s})-\tilde{h}(\mathfrak{s})
&=c^H(\mathfrak{s}I-A)^{-1}b-c^HV(\mathfrak{s}W^HV-W^HAV)^{-1}W^Hb\\
&=c^H(\mathfrak{s}I-A)^{-1}b-c^HV(\mathfrak{s}I-W^HAV)^{-1}W^Hb\\
&=\frac{1}{[\Lambda(\mathfrak{s})]^2}c^H(\mathfrak{s}I-A)^{-1}[\Lambda(A)]^2b,
\end{aligned}
\end{equation*}
where $\Lambda(\lambda):=\prod\limits_{i=1}^{m}(\lambda-\lambda_i)$
is the monic characteristic polynomial of $T_m=W^HAV$.
\end{theorem}

\begin{proof}
	Note that
\begin{equation*}
\begin{aligned}
r_b&=b-(\mathfrak{s}I-A)x_b,\\
(\mathfrak{s}I-A)^{-1}b-x_b&=(\mathfrak{s}I-A)^{-1}r_b,\\
c^H(\mathfrak{s}I-A)^{-1}b-c^Hx_b&=c^H(\mathfrak{s}I-A)^{-1}r_b,\\
 \mbox{similarly,}\quad  c^H(\mathfrak{s}I-A)^{-1}&=r_c^H(\mathfrak{s}I-A)^{-1}+x_c^H.\\
\end{aligned}
\end{equation*}

Since $x_c\in{\rm Krylov}(A^H,c,m)  \quad \mbox{and} \quad r_b\bot {\rm Krylov}(A^H,c,m),$
we know $x_c^Hr_b=0.$
By  
Lemma \ref{lm:Krylov_rb_rc}, we  obtain
\begin{equation*}
\begin{aligned}
e(\mathfrak{s})&=h(\mathfrak{s})-\tilde{h}(\mathfrak{s})
=c^H(\mathfrak{s}I-A)^{-1}b-c^Hx_b=c^H(\mathfrak{s}I-A)^{-1}r_b
=[r_c^H(\mathfrak{s}I-A)^{-1}+x_c^H]r_b\\
&=r_c^H(\mathfrak{s}I-A)^{-1}r_b
=c^H\frac{\Lambda(A)}{\Lambda(\mathfrak{s})}(\mathfrak{s}I-A)^{-1}\frac{\Lambda(A)}{\Lambda(\mathfrak{s})}b
=\frac{1}{[\Lambda(\mathfrak{s})]^2}c^H(\mathfrak{s}I-A)^{-1}[\Lambda(A)]^2b.
\end{aligned}
\end{equation*}
\end{proof}

\begin{remark}
	\label{rmk:assumption_cb_equal_1}
The relation \eqref{eqn:Krylov_lanczos} requires $c^Hb=1$, which is not a common occurrence in practice.
 Nevertheless, this has no impact on the result of  Theorem \ref{thm:arnoldi_speical_basis_result}.
The  coefficient $m_0:=c^Hb$ is indeed the first moment of 
$c^HA^ib$ \cite[(2.1)]{MR4311639}.
Note that $h(\mathfrak{s})=c^H(\mathfrak{s}I-A)^{-1}b=m_0 c^H(\mathfrak{s}I-A)^{-1}(b/m_0).$
If  $m_0\neq 1,$ then  the conclusions of  Theorem \ref{thm:arnoldi_speical_basis_result} remain valid, after
our  applying Theorem \ref{thm:arnoldi_speical_basis_result}
onto $(A,b/m_0,c)$ and multiplying the  resulting expression by $m_0$.
\end{remark}


The relation $e(\mathfrak{s})=r_c^H(\mathfrak{s}I-A)^{-1}r_b$ is known in Grimme's PhD thesis\cite[Theorem 5.1]{grimme1997krylov}.
The traditional result shows that $e(\mathfrak{s})=O(1/\mathfrak{s}^{2m+1})$ (e.g.\cite[Section 11.2.1]{MR2155615} \cite[(3.3.26)]{MR3024841}). We now present a clear and explicit formulation of the error.

\subsection{ {\it Step 2:} General subspaces (rational Krylov subspaces) and special bases}
\label{section:special_basis}
\subsubsection{Biorthogonal bases of $\CK(A,b,k_b,m_b)$ and $\CK(A^H,c,k_c,m_c)$}\label{sec:isomorphism_section}
The rational Krylov subspace can be obtained by a special standard Krylov subspace (e.g.\cite[Lemma 4.2(a)]{guttel2010rational}).
%
%
%
%
%
%
%
With \eqref{eqn:cks_basis_b}, we have
\begin{equation}\label{eq:CK_and_Krylov_subspaces_relation}
\begin{aligned}
\cks(A,b,k_b,m_b)&=\Krylov(A,\varphi(A)^{-1}b,l), &\cks(A,b,k_b,m_b+1)&=\Krylov(A,\varphi(A)^{-1}b,l+1),\\
\cks(A^H,c,k_c,m_c)&=\Krylov(A^H,\psi(A)^{-H}c,l),&\cks(A^H,c,k_c,m_c+1)&=\Krylov(A^H,\psi(A)^{-H}c,l+1).\\
\end{aligned}
\end{equation}


We use the Lanczos biorthogonalization procedure of $A$, $\varphi(A)^{-1}b$ and $\psi(A)^{-H}c$ to form the biorthogonal bases of $\Krylov(A,\varphi(A)^{-1}b,l+1)$ and $\Krylov(A^H,$ $\psi(A)^{-H}c,l+1)$.
%
%
%
%
%
%
%
%
%
%
%
%
%
The relations are summarized as follows.

{\small
\begin{equation}\label{eqn:generalized_nonsymmetric_Lanczos}
\begin{aligned}
	\widehat v_1&=\varphi(A)^{-1}b,
	\quad 
	\widehat w_1=\psi(A)^{-H}c,
	\quad 
	\widehat w_1^H\widehat v_1=1,\\
A\widehat V_{l}&=\widehat V_{l+1}\underline{\widehat T}_{l}, \quad \span(\widehat V_{l})={\rm Krylov}(A,\varphi(A)^{-1}b,l),\quad
\span(\widehat V_{l+1})={\rm Krylov}(A,\varphi(A)^{-1}b,l+1),\\
A^H \widehat W_{l}&=\widehat W_{l+1}\underline{\widehat K}_{l}, \quad \span(\widehat W_{l})={\rm Krylov}(A^H,\psi(A)^{-H}c,l),
\quad \span(\widehat W_{l+1})={\rm Krylov}(A^H,\psi(A)^{-H}c,l+1),\\
\widehat W_{l}^H\widehat V_{l}&=I,\quad \widehat W_{l+1}^H\widehat V_{l+1}=I, \quad
\widehat W_{l}^H A\widehat V_{l}=\widehat T_{l},
\quad
\widehat V^HA^H\widehat W_{l}=\widehat K_{l},
\quad
\widehat T_{l}=\widehat K_{l}^H,\\
%
\underline{\widehat T}_{l+1}&=
\left[ \begin{array}{ccccc}
\widehat \alpha_1& \widehat \beta_2 & &&\\
\widehat \gamma_2 & \widehat \alpha_2 &\widehat \beta_3 &&\\
 &   \ddots & \ddots&\ddots&\\
&&\widehat \gamma_{l-1}&\widehat \alpha_{l-1}&\widehat \beta_l\\
&&&\widehat \gamma_l&\widehat \alpha_l\\
\hline
&&&&\widehat \gamma_{l+1}\\
\end{array}
\right]
=\left[ \begin{array}{c}
	\widehat T_m\\
	\hline
	\widehat \gamma_{m+1}e_m^H
\end{array}\right].\\
\end{aligned}
\end{equation}
}

As in Section \ref{subsect:infinty_shifts_special_basis}, we adopt the same notations for
$x_b,x_c,r_b,r_c$ and their rational function expressions.
%
%
%
The two-sided projection (Petrov-Galerkin) conditions  are imposed as follows:
\begin{equation}\label{eqn:x_b_in_CK_projection}
\begin{aligned}
&x_b\in \CK(A,b,k_b,m_b), 
\qquad  r_b=b-(\mathfrak{s}I-A)x_b \in \CK(A,b,k_b,m_b+1) ,\\
\st &\quad  r_b \bot \CK(A^H,c,k_c,m_c),\\
\end{aligned}
\end{equation}
and
\begin{equation*}
\begin{aligned}
&x_c \in \CK(A^H,c,k_c,m_c),
\qquad  r_c=c-(\mathfrak{s}I-A)^Hx_c \in \CK(A^H,c,k_c,m_c+1),\\
\st  &\quad r_c\bot \CK(A,b,k_b,m_b).\\
\end{aligned}
\end{equation*}

Thus, we get $\tilde{h}(\mathfrak{s})=x_c^Hb=c^Hx_b,$ where 
\begin{equation*}
\begin{aligned}
x_b&=\widehat V(\mathfrak{s}I-\widehat W^HA\widehat V)^{-1}\widehat W^Hb,
\quad 
x_c=\widehat W(\mathfrak{s}I-\widehat W^HA\widehat V)^{-H}\widehat V^Hc.\\
\end{aligned}
\end{equation*}

By \eqref{eq:CK_and_Krylov_subspaces_relation}, the rational function expressions are as follows:
\begin{equation*}
	\begin{aligned}
		x_b&=\chi_b(A,\mathfrak{s})b, \quad \chi_b(\lambda,\mathfrak{s}) \in \mathbb{P}_{l-1}(\lambda)/\varphi(\lambda),\\
		r_b&=\gamma_b(A,\mathfrak{s})b,\quad \gamma_b(\lambda,\mathfrak{s}) \in \mathbb{P}_{l}(\lambda)/\varphi(\lambda).\\
	\end{aligned}
\end{equation*}
Their relations are given by the following formulas.
\begin{equation}\label{eqn:relation_r_b_x_b_CK_case}
	\begin{aligned}
		r_b&=b-(\mathfrak{s}I-A)x_b,\\
		\gamma_b(\lambda,\mathfrak{s})&=1-(\mathfrak{s}-\lambda)\chi_b(\lambda,\mathfrak{s}),\\
		\chi_b(\lambda,\mathfrak{s})&=\frac{1-\gamma_b(\lambda,\mathfrak{s})}{\mathfrak{s}-\lambda}.\\
	\end{aligned}
\end{equation}
Therefore, we still have the 
constraint condition:
$\gamma_b(\mathfrak{s},\mathfrak{s})=1$.
%

\subsubsection{The derivation}
Prior to the presentation of the final expressions of $r_b$ and $r_c$ derived from the two-sided projection method,
it is noteworthy to mention that the residual expression from the one-sided projection method has been derived by other approaches
(e.g. \cite[(3.5)]{MR2485440}\cite[(7.17)]{guttel2010rational}).
The result has been employed for the purposes of theoretical analysis \cite{MR2854603} and algorithms development 
\cite{MR2858340,MR3570279}.

\begin{lemma}\label{lm:rational_expression_of_residuals}
	Suppose relations \eqref{eqn:generalized_nonsymmetric_Lanczos} hold.
	Set
	$g_b(\mathfrak{s}):=\Lambda(\mathfrak{s})/\varphi(\mathfrak{s})$ and 
	$g_c(\mathfrak{s}):=\Lambda(\mathfrak{s})/\psi(\mathfrak{s}),$
	where $\Lambda(\lambda)$ is the monic characteristic polynomial of  $\widehat T_l=\widehat W^HA\widehat V$.
	Then, it holds that 
\begin{equation*}
\begin{aligned}
\gamma_b(\lambda,\mathfrak{s})&=\frac{g_b(\lambda)}{g_b(\mathfrak{s})},\quad& r_b&=\gamma_b(A,\mathfrak{s})b, \\
\gamma_c(\lambda,\mathfrak{s})&=\frac{g_c(\lambda)}{g_c(\mathfrak{s})},\quad &r_c^H&=c^H\gamma_c(A,\mathfrak{s}).\\
\end{aligned}
\end{equation*}
\end{lemma}
\begin{proof}
 After translating \eqref{eqn:x_b_in_CK_projection}
  into the standard Krylov subspace language with \eqref{eq:CK_and_Krylov_subspaces_relation}, we obtain
\begin{equation*}
\begin{aligned}
&x_b\in \Krylov(A,b,l), 
\qquad  r_b=b-(\mathfrak{s}I-A)x_b \in \Krylov(A,\varphi(A)^{-1}b,l+1), \\
s.t. &\quad  r_b \bot \Krylov(A^H,\psi(A)^{-H}c,l). \\
\end{aligned}
\end{equation*}

By Lemma \ref{lm:Krylov_Cayley}, we directly obtain
$r_b=\rho \Lambda(A)[\varphi(A)^{-1}b]=\rho g_b(A)b,$
where $\Lambda(\lambda)$  is  the monic characteristic polynomial of  $\widehat T_{l}$.
As $\mathfrak{s}$ varies, so does the coefficient $\rho$. Thus, we can regard $\rho$ as a function of the  variable $\mathfrak{s}$. With a new symbol  $\rho(\mathfrak{s}),$ we obtain
\begin{equation*}
\begin{aligned}
\gamma_b(\lambda,\mathfrak{s})= \rho(\mathfrak{s}) \frac{\Lambda(\lambda) }{\varphi(\lambda)}=\rho(\mathfrak{s})g_b(\lambda).
\end{aligned}
\end{equation*}

Given  the constraint condition $\gamma_b(\mathfrak{s},\mathfrak{s})=1$, it follows that $\rho(\mathfrak{s})=1/g_b(\mathfrak{s}).$
Hence, $\gamma_b(\lambda,\mathfrak{s})=g_b(\lambda)/g_b(\mathfrak{s}).$
Similarly, we shall obtain the result about $r_c$.
%
\end{proof}

\begin{remark}\label{rmk:solution_expression}
By   relation \eqref{eqn:relation_r_b_x_b_CK_case}, 
we   obtain a new expression of  $\tilde{h}(\mathfrak{s})=c^Hx_b=x_c^Hb:$
\begin{equation*}
\begin{aligned}
\tilde{h}(\mathfrak{s})&=c^H\chi_b(A,\mathfrak{s})b=c^H\chi_c(A,\mathfrak{s})b,\\
\chi_b(\lambda,\mathfrak{s})&=\frac{1-\gamma_b(\lambda,\mathfrak{s})}{\mathfrak{s}-\lambda},
\qquad
\chi_c(\lambda,\mathfrak{s})=\frac{1-\gamma_c(\lambda,\mathfrak{s}) }{\mathfrak{s}-\lambda}.\\
\end{aligned}
\end{equation*}
\end{remark}

\begin{theorem}\label{thm:general_subspace_speical_basis}
Suppose the Lanczos biorthogonalization procedure of $\left[A,\varphi(A)^{-1}b, \psi(A)^{-H}c\right]$ does not break down until
$(l+1)$th iteration.
The formed matrices satisfy \eqref{eqn:generalized_nonsymmetric_Lanczos}. 
Then, it holds that
\begin{equation*}
\begin{aligned}
h(\mathfrak{s})-\tilde{h}(\mathfrak{s})
&=c^H(\mathfrak{s}I-A)^{-1}b-c^H\widehat V(\mathfrak{s}I-\widehat W^HA\widehat V)^{-1}\widehat W^Hb\\
&=r_c^H(\mathfrak{s}I-A)^{-1}r_b\\
&=\frac{1}{g_b(\mathfrak{s})g_c(\mathfrak{s})}c^Hg_c(A)(\mathfrak{s}I-A)^{-1}g_b(A)b,\\
\end{aligned}
\end{equation*}
where $\Lambda(\lambda)$ is the monic characteristic polynomial of  $\widehat T_l=\widehat W^HA\widehat V$.
\end{theorem}

The proof is straightforward.  By substituting the result of Lemma  \ref{lm:rational_expression_of_residuals} into 
$e(\mathfrak{s})=r_c^H(\mathfrak{s}I-A)^{-1}r_b$, we obtain the desired result.
The assumption $c^H\psi(A)^{-1}\varphi(A)^{-1}b=1$
in  \eqref{eqn:generalized_nonsymmetric_Lanczos} also does not affect the result. 
The discussion is  analogous to an earlier discussion of 
$c^Hb=1$ for Theorem \ref{thm:arnoldi_speical_basis_result}, as stated in Remark \ref{rmk:assumption_cb_equal_1}.
In essence, we obtain the  explicit error formula by utilizing   special bases $\widehat V$ and $\widehat W$ of  general subspaces $\CK(A,b,k_b,m_b)$ and $\CK(A^H,c,k_c,m_c)$.

\subsection{{\it Step 3:} General subspaces and general bases}\label{subsection:main_result_proof}
We give the final  proof of Theorem \ref{thm:main_result_1}.

\begin{proof}
Let $\widehat{V}$ and $\widehat{W}$
be the special bases used in 
Theorem \ref{thm:arnoldi_speical_basis_result}.
They satisfy 
$\span(\widehat{V}_l)=\CK(A,b,k_b,m_b),$ 
$\span(\widehat{W}_l)=\CK(A^H,c,k_c,m_c)$ and 
$\widehat{W}^H\widehat{V}=I.$

Because of $\span(V)=\CK(A,b,k_b,m_b)$ and $\span(W)=\CK(A^H,c,k_c,m_c)$,
we obtain the relations $V=\widehat{V}R$
and $W=\widehat{W}S,$
%
where both $R$ and $S$ are  nonsingular transformation matrices.
Then, it holds that
\begin{equation*}\label{eqn:h_tilde_transformation}
\begin{aligned}
\tilde{h}(\mathfrak{s})&=c^HV(\mathfrak{s}W^HV-W^HAV)^{-1}W^Hb=c^H\widehat{V}R( zS^H\widehat{W}^H\widehat{V}R-S^H\widehat{W}^HA\widehat{V}R)^{-1}S^H\widehat{W}^Hb\\
&=c^H\widehat{V}( z\widehat{W}^H\widehat{V}-\widehat{W}^HA\widehat{V})^{-1}\widehat{W}^Hb
=c^H\widehat{V}( zI-\widehat{W}^HA\widehat{V})^{-1}\widehat{W}^Hb.
\end{aligned}
\end{equation*}

By Theorem~\ref{thm:general_subspace_speical_basis}, we obtain
\begin{equation*}
\begin{aligned}
h(\mathfrak{s})-\tilde{h}(\mathfrak{s})
=\frac{1}{g_b(\mathfrak{s})g_c(\mathfrak{s})}c^Hg_c(A)(\mathfrak{s}I-A)^{-1}g_b(A)b,
\end{aligned}
\end{equation*}
where $\lambda_i (i=1,2,\ldots, l)$ are eigenvalues of $\widehat{W}^HA\widehat{V}.$
In this formula, the only elements related to the bases $\widehat{V}$ and $\widehat{W}$ are the eigenvalues.
Let us now see how to relate these eigenvalues to the bases $V$ and $W$.
Obviously, we have
$\eig(\widehat{W}^HA\widehat{V})=\eig(S^{-H}W^HAVR^{-1})
=\eig((S^HR)^{-1}W^HAV)=\eig((W^HV)^{-1}W^HAV).$
%
\end{proof}



Researchers have already utilized the rational biorthogonal bases to design MOR
\cite{MR3473700,MR3592472,frangos2007adaptive,MR2501562}.
For a theoretical analysis of rational biorthogonal bases, see \cite{MR4040919}.
Since the error formula is independent of the bases, 
 it is recommended that  both $V$ and $W$ be orthonormal in practical  implementation. They satisfy
 $V_l^HV_l=I,  W_l^HW_l=I$ and
 $W_l^HV_l \neq I,$ where 
 $\span(V_l)=\CK(A,b,k_b,m_b)$ 
 and $\span(W_l)=\CK(A^H,c,k_c,m_c).$
%
%
To obtain the orthonormal bases,
we can apply Arnoldi-like procedure onto bases (\ref{eqn:cks_basis_b}). 
If the shifts sets $\mathbb{T}$ and $\mathbb{S}$ are given,  a more convenient approach 
is to call  \emph{A Rational Krylov Toolbox for MATLAB} \cite{berljafa2014rational}.


With  the new notation 
\begin{equation}\label{eqn:G_two_form}
\begin{aligned}
G_{\rm two}(\lambda)
:=g_b(\lambda)g_c(\lambda)=\frac{\Lambda(\lambda)}{\varphi(\lambda)}\frac{\Lambda(\lambda)}{\psi(\lambda)},
\end{aligned}
\end{equation}
the error formula is simplified as follows.
\begin{equation}\label{eqn:G_two_error_formula_without_E}
\begin{aligned}
e(\mathfrak{s})=h(\mathfrak{s})-\tilde{h}(\mathfrak{s})
&=\frac{1}{G_{\rm two}(\mathfrak{s})}c^H(\mathfrak{s}I-A)^{-1}G_{\rm two}(A)b.
\end{aligned}
\end{equation}

\subsection{Two further explanations for the MOR  error}\label{sect:two_proofs_by_Hermitian formula}
The explicit identification of the error formula allows for the deeper comprehension of the fundamental characteristics of the MOR error. 
We discover  two further explanations for the MOR error, which are closely related to the two parameters of the resolvent function $\mathcal {H}(\lambda,\mathfrak{s})=1/(\mathfrak{s}-\lambda)$. 
The MOR error can be classified as either an interpolation remainder with respect to variable $\mathfrak{s}$ or a quadrature remainder with respect to $\lambda$.
In order to provide a concise presentation, the demonstration will primarily focus on how the terms within MOR correspond to those within the classical interpolation (or quadrature) theory.
The omitted proofs can be located in the preceding version of the manuscript \cite{2verlin2021transferfunctioninterpolationremainder}.

In interpolation theory, the Lagrange-type remainder can be applied if all of the nodes are real (e.g.\cite[Theorem 1]{MR4433120}).
Given that all of the shifts and Ritz values in our problem are complex, it is vital to utilize  alternative theoretical tools.
 Our theoretical tool is the Hermitian remainder or the divided difference remainder.
The Hermitian formula represents the Cauchy integral form  of the interpolation remainder (e.g.\cite[(8)]{MR2241723}\cite[Page 59]{gaier1987lectures}).
After transforming the polynomial function to the rational function \cite[Page 33]{guttel2010rational}, we can derive an  error formula for  rational function interpolation, which is named as the Walsh-Hermite formula in \cite[Page 19]{MR3095912}.

\begin{lemma}\label{lm:hermite_remainder_intergral}
	{\rm (Hermite)}
	 	Suppose the boundary $\Gamma$ of $\Sigma$ consists of finitely many rectifiable Jordan curves 
	with positive orientation relative to $\Sigma$, and suppose $M(\mathfrak{s})$ is analytic in $\Sigma$ and  
	continuous in $\Sigma \cup \Gamma$. 
	Suppose  the interpolation  conditions hold, i.e.,
	$M(\alpha_i)=P_{k-1}(\alpha_i)$ for 
	$\alpha_i \in \Sigma  (i=1,2,\ldots,k) .$
    Write
    $\pi^{\alpha}_k(\mathfrak{s})=\prod_{i=1}^k(\mathfrak{s}-\alpha_i).$
    Then,  it holds that
    \begin{equation*}
    	\begin{aligned}
    		M(\mathfrak{s})-P_{k-1}(\mathfrak{s})
    		=\frac{1}{2\pi \iota}\oint_{\Gamma}\frac{\pi^{\alpha}_k(\mathfrak{s})}{\pi^{\alpha}_k(\zeta)}\frac{M(\zeta)}{\zeta-\mathfrak{s}}d\zeta.\\
    	\end{aligned}
    \end{equation*}
	
\end{lemma}

\begin{lemma}\label{lm:rational_Hermitian_remainder}
{\rm (Walsh-Hermite)}
Let $Q_{k-1,k}(\mathfrak{s})\in \mathbb{Q}_{k-1,k}(\mathfrak{s})$ be a rational  interpolating function  of $N(\mathfrak{s})$  with poles $\beta_i$ and interpolation nodes  $\alpha_i$.
Write 
\begin{equation*}
\begin{aligned}
G^{\alpha,\beta}_k(\mathfrak{s})
:=\prod\limits_{i=1}^k\frac{z-\alpha_i}{z-\beta_i}
=\frac{(\mathfrak{s}-\alpha_1)(\mathfrak{s}-\alpha_2)\cdots(\mathfrak{s}-\alpha_k)}
{(\mathfrak{s}-\beta_1)(\mathfrak{s}-\beta_2)\cdots(\mathfrak{s}-\beta_k)}.
\end{aligned}
\end{equation*}
	 	Suppose the boundary $\Gamma$ of $\Sigma$ consists of finitely many rectifiable Jordan curves 
with positive orientation relative to $\Sigma$, and suppose  $N(\mathfrak{s})$ is analytic in $\Sigma$ and  
continuous in $\Sigma \cup \Gamma$. 
%
If interpolation nodes $\alpha_i (i=1,2,\ldots,k)$  are in $\Sigma$, 
then it holds that
\begin{equation*}
\begin{aligned}
N(\mathfrak{s})-Q_{k-1,k}(\mathfrak{s})
=\frac{1}{2\pi i}\oint_{\Gamma}\frac{G^{\alpha,\beta}_k(\mathfrak{s})}{G^{\alpha,\beta}_k(\zeta)}\frac{N(\zeta)}{\zeta-\mathfrak{s}}d\zeta.\\
\end{aligned}
\end{equation*}

\end{lemma}

\begin{proof}
	Let 
	$\pi^{\beta}_k(\mathfrak{s}):=\prod_{i=1}^k(\mathfrak{s}-\beta_i),$
	$\widetilde N(\mathfrak{s}):=
	N(\mathfrak{s})\pi^{\beta}_k(\mathfrak{s})$
	and $P_{k-1}(\mathfrak{s}):=
	Q_{k-1,k}(\mathfrak{s})\pi^{\beta}_k(\mathfrak{s}).$
	By the rational interpolation condition, we  obtain
	$\widetilde N(\alpha_i)=P_{k-1}(\alpha_i)$ for 
	$\alpha_i \in \Sigma  (i=1,2,\ldots,k) .$
	Thus, Lemma \ref{lm:hermite_remainder_intergral} is applied.
	    \begin{equation}\label{eq:Nzpi_interpolation_intergral}
		\begin{aligned}
			N(\mathfrak{s})\pi^{\beta}_k(\mathfrak{s})-P_{k-1}(\mathfrak{s})=
			\widetilde N(\mathfrak{s})-P_{k-1}(\mathfrak{s})
			=\frac{1}{2\pi \iota}\oint_{\Gamma}\frac{\pi^{\alpha}_k(\mathfrak{s})}{\pi^{\alpha}_k(\zeta)}\frac{\widetilde N(\zeta)}{\zeta-\mathfrak{s}}d\zeta
			=\frac{1}{2\pi \iota}\oint_{\Gamma}\frac{\pi^{\alpha}_k(\mathfrak{s})}{\pi^{\alpha}_k(\zeta)}\frac{N(\zeta)\pi^{\beta}_k(\zeta)}{\zeta-\mathfrak{s}}d\zeta.\\
		\end{aligned}
	\end{equation}

The proof  is completed 
after  our dividing   \eqref{eq:Nzpi_interpolation_intergral} by $\pi^{\beta}_k(\mathfrak{s}).$
\end{proof}

In the next two subsections, 
we apply Lemma  \ref{lm:rational_Hermitian_remainder} onto the resolvent function $\mathcal{H}(\lambda,\mathfrak{s})=1/(\mathfrak{s}-\lambda)$ with respect to  variables $\mathfrak{s}$ and $\lambda$, respectively.
Note that the interpolation nodes and poles of  these  two cases are opposite.
The key step in the following poofs  is the interchange between
the problem resolvent function $\mathcal{H}(\lambda,\mathfrak{s})$ and the resolvent function in the Cauchy integral.
Therefore, if  other problems  (e.g. \cite[Section 4.2, Section 4.3]{MR3095912}) lack a  resolvent function term inside, we may not expect to acquire  an explicit error formula.

%

\subsubsection{Variable $\mathfrak{s}$}
To make use of Lemma \ref{lm:rational_Hermitian_remainder},
we  substitute $h(\mathfrak{s}) \in \mathbb{Q}_{n-1,n}(\mathfrak{s})$ 
and $\tilde{h}(\mathfrak{s})\in \mathbb{Q}_{l-1,l}(\mathfrak{s})$  into $N(\mathfrak{s})$ and $Q_{k-1,k}(\mathfrak{s})$, respectively. 
For simplicity, we assume that
$k_b=k_c=l$ and that all of the 
shifts $t_i,s_i$ (from $\varphi(\lambda),\psi(\lambda)$) are finite.
Consider  a region $\Sigma$ such that $\eig(A)\cup \{\lambda_i\} \subseteq \Sigma$ and $ \{s_i,t_i\} \subseteq \Sigma^{-}$.
We observe that   $\tilde{h}(\mathfrak{s})$  is a rational  interpolating function  of $h(\mathfrak{s})$  with poles $\lambda_i$(doubled) and interpolation nodes $t_i,s_i.$ 
Thus, we substitute 
$G^{\alpha,\beta}_k(\mathfrak{s})=1/G_{\rm two}(\mathfrak{s})$ into Lemma \ref{lm:rational_Hermitian_remainder}, where
 $G_{\rm two}(\mathfrak{s})$ is defined by (\ref{eqn:G_two_form}).
 It is easy to check that 
 $h(\mathfrak{s})$ is analytic on  $\Sigma^{-}.$ By Lemma \ref{lm:rational_Hermitian_remainder},  for $\mathfrak{s}\in \Sigma^{-}$, it holds that
 \begin{equation*}
 	\begin{aligned}
 		h(\mathfrak{s})-\tilde{h}(\mathfrak{s})
 		&=\frac{1}{2\pi \iota}\oint_{\Gamma^{-}}\frac{G^{\alpha,\beta}_k(\mathfrak{s})}{G^{\alpha,\beta}_k(\zeta)}\frac{h(\zeta)}{\zeta-\mathfrak{s}}d\zeta
 		=\frac{1}{2\pi \iota}\oint_{\Gamma^{-}}\frac{G_{two}(\zeta)}{G_{\rm two}(\mathfrak{s})}\frac{c^H(\zeta I-A)^{-1}b}{\zeta-\mathfrak{s}}d\zeta\\
 		&=\frac{1}{G_{\rm two}(\mathfrak{s})}\frac{1}{2\pi \iota}c^H\left[\oint_{\Gamma^{-}}\frac{G_{\rm two}(\zeta)}{\zeta-\mathfrak{s}}(\zeta I-A)^{-1}d\zeta \right]b\\
 		&=\frac{1}{G_{\rm two}(\mathfrak{s})}\frac{1}{2\pi \iota}c^H\left[\oint_{\Gamma}\frac{G_{\rm two}(\zeta)}{\mathfrak{s}-\zeta}(\zeta I-A)^{-1}d\zeta \right]b\\
 		&=\frac{1}{G_{\rm two}(\mathfrak{s})}c^HG_{\rm two}(A)(\mathfrak{s}I-A)^{-1}b.
 	\end{aligned}
 \end{equation*}
In the last equality, we utilize the Cauchy integral definition for  functions of matrices(e.g.\cite[Definition 1.11]{MR2396439}), given that  $G_{\rm two}(\zeta)/(\mathfrak{s}-\zeta)$ with respect to $\zeta$ is analytic on $\Sigma \supseteq \eig(A) \cup \{\lambda_i\}  $.

The following serves to provide an outline of the explanation with respect to variable $\mathfrak{s}$. 
In essence, it can be regarded as an interpolation formula.
The functions $h(\mathfrak{s})$ and $\tilde{h}(\mathfrak{s})$ represent the  interpolated and the interpolating  functions, respectively. 
The  moments matching  theory  states that the interpolation nodes are the shifts $t_i,s_i$ \cite{MR2013459}.
Following the discovery that the order of the poles, $\lambda_i$, is two, it becomes evident that the MOR error  $e(\mathfrak{s})$ can be regarded as the remainder of the interpolation formula.

\subsubsection{Variable $\lambda$}\label{subsect:Hermitian_remainder_lambda}
With respect to the parameter $\lambda$, the initial step will be to construct the classical Hermitian interpolating polynomial, incorporating either the Hermitian remainder or the divided difference remainder.
Subsequently, the Gauss quadrature is employed.
We shall substitute $\widehat{\mathcal{H}}(\lambda)=\varphi(\lambda)\psi(\lambda)/(\mathfrak{s}-\lambda)$ into $\widetilde N(\mathfrak{s})$ in \eqref{eq:Nzpi_interpolation_intergral}. 
Now, the variable of interest  is $\lambda$. 
For simplicity, we assume that there exists  a set $\Sigma$ satisfying $\eig(A)\cup \{\lambda_i\} \subseteq \Sigma$ and $ \{s_i,t_i\} \subseteq \Sigma^{-}$. 
In order to provide a concise explanation, it is assumed that  all of $\{\lambda_i\}$ are distinct.
Otherwise, the Gauss quadrature process is known to become more complicated\cite[(4.1)]{MR4311639}.

\begin{proposition}\label{prop:Hermitian_interpolation_of_lamda}
	Let $\widetilde{\mathcal P}_{2l-1}(\lambda)\in \mathbb{P}_{2l-1}(\lambda)$ be the Hermitian interpolating polynomial of $\widehat{\mathcal{H}}(\lambda)$  on the interpolation nodes   $\lambda_i$ (doubled). 
	Then, the interpolation remainder can be expressed as follows: 
			\begin{eqnarray}
		&	\widehat{\mathcal{H}}(\lambda)-\widetilde{\mathcal P}_{2l-1}(\lambda)
			=[\Lambda(\lambda)]^2\widehat{\mathcal{H}}[\lambda_1,\lambda_1,\lambda_2,\lambda_2,\ldots,\lambda_l,\lambda_l,\lambda], \nonumber \\
		&\widehat{\mathcal{H}}[\lambda_1,\lambda_1,\lambda_2,\lambda_2,\ldots,\lambda_l,\lambda_l,\lambda]
       =\left\{\frac{\varphi(\mathfrak{s})\psi(\mathfrak{s})}{[\Lambda(\mathfrak{s})]^2}\right\}\frac{1}{\mathfrak{s}-\lambda}. \label{eqn:divied_difference_finite}
       \end{eqnarray}
\end{proposition}

\begin{proof}
%
%
%
%
 One of the explicit expressions of  $\widetilde{\mathcal P}_{2l-1}(\lambda)$ is given as follows (e.g.\cite[Theorem 3.9]{burden1993numerical}).
		\begin{eqnarray*}
		\widehat{\mathcal{H}}(\lambda_i)&=&\widetilde{\mathcal P}_{2l-1}(\lambda_i),
		\quad
		\widehat{\mathcal{H}}'(\lambda_i)=\widetilde{\mathcal P}_{2l-1}'(\lambda_i),
		\quad
		i=1,2,\ldots,l,\nonumber\\
		\widetilde{\mathcal P}_{2l-1}(\lambda)
		&=&\sum\limits_{i=1}^{l}
		\widehat{\mathcal{H}}(\lambda_i)
		\{\widehat{\mathcal L}_i^2(\lambda)[1-2(\lambda-\lambda_i)]\widehat{\mathcal L}_i'(\lambda)\}
		+\sum\limits_{i=1}^{l}
		\widehat{\mathcal{H}}'(\lambda_i)
		[(\lambda-\lambda_i)\widehat{\mathcal L}_i^2(\lambda)],
		\nonumber
		\end{eqnarray*}
where $\widehat{\mathcal L}_i(\lambda)$ is the Lagrange basis polynomial 
\begin{equation*}
	\begin{aligned}
		\widehat{\mathcal L}_i(\lambda)
		=\prod\limits_{j=1,j\neq i}^l\frac{\lambda-\lambda_j}{\lambda_i-\lambda_j}.
	\end{aligned}
\end{equation*}
This remainder formula is, in fact, a quantity equality, which permits a relatively straightforward proof.
A proof based on the meanings of the divided difference can be found in  \cite[Lemma 2.15 and (2.24)]{2verlin2021transferfunctioninterpolationremainder}.
%
%
%

An additional proof is provided by using Lemma~\ref{lm:hermite_remainder_intergral}.
In this proof, an observation is made regarding the interchange between the problem resolvent function  $\mathcal{H}(\lambda,\mathfrak{s})$ and the resolvent function in the Cauchy integral.
\begin{equation}\label{eqn:H_hat_lambda_cauchy_integration}
\begin{aligned}
&\widehat{\mathcal{H}}(\lambda)-\widetilde{\mathcal P}_{2l-1}(\lambda)
=\frac{1}{2\pi \iota}\oint_{\Gamma}\frac{[\Lambda(\lambda)]^2}{[\Lambda(\zeta)]^2}
\frac{\widehat{\mathcal{H}}(\zeta)}{\zeta-\lambda}d\zeta
=[\Lambda(\lambda)]^2\frac{1}{2\pi \iota}\oint_{\Gamma}\frac{1}{[\Lambda(\zeta)]^2}
\frac{\varphi(\zeta)\psi(\zeta)}{\mathfrak{s}-\zeta}\frac{1}{\zeta-\lambda}d\zeta\\
&=[\Lambda(\lambda)]^2\frac{1}{2\pi \iota}\oint_{\Gamma}\frac{1}{[\Lambda(\zeta)]^2}
\frac{\varphi(\zeta)\psi(\zeta)}{\lambda-\zeta}\frac{1}{\zeta-\mathfrak{s}}d\zeta
=[\Lambda(\lambda)]^2\frac{1}{2\pi \iota}\oint_{\Gamma^{-}}\left\{\frac{1}{[\Lambda(\zeta)]^2}
\frac{\varphi(\zeta)\psi(\zeta)}{\zeta-\lambda}\right\}\frac{1}{\zeta-\mathfrak{s}}d\zeta\\
&=[\Lambda(\lambda)]^2\frac{1}{[\Lambda(\mathfrak{s})]^2}
\frac{\varphi(\mathfrak{s})\psi(\mathfrak{s})}{\mathfrak{s}-\lambda}
=[\Lambda(\lambda)]^2\widehat{\mathcal{H}}[\lambda_1,\lambda_1,\lambda_2,\lambda_2,\ldots,\lambda_l,\lambda_l,\lambda].\\
\end{aligned}
\end{equation}
The Cauchy integral formula is employed in  the penultimate equality,
given that the function in the braces
with respect to $\zeta$ is analytic on $\Sigma^{-}$.
%
%
%
%
\end{proof}

Once the interpolation phase is complete, the Gauss  quadrature stage begins, which is now expressed as a linear functional.
We begin by presenting two propositions pertaining to the classical result in Gauss quadrature.



\begin{proposition}\label{thm:moment_mathcing_weighted}
	Suppose   \eqref{eqn:generalized_nonsymmetric_Lanczos} hold.
	Then, it holds that
	$c^H\psi(A)^{-1}\mathcal {P}(A)\varphi(A)^{-1}b=e_1^H\mathcal {P}(\widehat{T}_l)e_1 $ for any $\mathcal {P}(\lambda)\in \mathbb{P}_{2l-1}$.
\end{proposition}
\begin{proof}
	By (\ref{eqn:generalized_nonsymmetric_Lanczos}), the relating matrices are formed by  the Lanczos biorthogonalization procedure of $A$, $\varphi(A)^{-1}b$ and  $\psi(A)^{-H}c$.
	Accordingly, the proof is completed by means of a direct application of the classical result (\! \cite[Theorem 2]{Freund1993}) onto $A, \varphi(A)^{-1}b$ and  $\psi(A)^{-H}c.$
	%
	%
\end{proof}

\begin{proposition}\label{th:algebraic_precision}
	Suppose relations \eqref{eqn:generalized_nonsymmetric_Lanczos}
	hold. Then, it holds that $c^H\mathcal{Q}(A)b=c^H\widehat V\mathcal{Q}(\widehat W^HA\widehat V)\widehat W^Hb$
	for any  ${\mathcal{Q}(\lambda)\in \mathbb{P}_{2l-1}(\lambda)/[\varphi(\lambda)\psi(\lambda)}]$.
\end{proposition}

\begin{proof}
	Similar to Lemma \ref{lm:polynomial_expression_taub}, 
	for the polynomial $\grave{\tau}_j(\lambda):=\grave{a}_j^j\lambda^j+\cdots+\grave{a}_1^j\lambda+\grave{a}_0^j,$ we obtain 
	\begin{equation*}
		\left\{
		\begin{aligned}
			\grave{\tau}_j(A)\varphi(A)^{-1}b&=\widehat{V}_l\grave{\tau}_j(\widehat{T}_l)e_1, \quad& j<l,\\
			\grave{\tau}_l(A)\varphi(A)^{-1}b&=\widehat{V}_l\grave{\tau}_l(\widehat{T}_l)e_1+\grave{a}_l^l\grave{\zeta}_l\widehat{v}_{l+1},& j=l,
		\end{aligned}
		\right.
	\end{equation*}
	with $	\grave{\zeta}_l=\widehat \gamma_{l+1}\cdots \widehat \gamma_3 \widehat \gamma_2.$
	Setting $\grave{\tau}_j(\lambda)=\varphi_{k_b}(\lambda),$ we obtain the following result:
	\begin{equation*}
		\left\{
		\begin{aligned}
			\varphi_{k_b}(A)\varphi_{k_b}(A)^{-1}b&=\widehat{V}_l\varphi_{k_b}(\widehat{T}_l)e_1,  & & k_b<l,\\
			\varphi_{k_b}(A)\varphi_{k_b}(A)^{-1}b&=\widehat{V}_l\varphi_{k_b}(\widehat{T}_l)e_1+\grave{a}_l\grave{\zeta}_l\widehat{v}_{l+1},&&k_b=l,m_b=0.\\
		\end{aligned}
		\right.
	\end{equation*}
	Note that $\widehat{W}^H_l\widehat{V}_l=I,\widehat{W}^H_l\widehat{v}_{l+1}=0$
	and $W^H\varphi(A)^{-1}b=e_1.$
	For either $k_b < l$ or $k_b = l$, it holds that 
	\begin{equation}\label{eqn:e1_b_right}
		\begin{aligned}
			\widehat{W}^Hb=\widehat{W}^H\varphi(A)\varphi(A)^{-1}b&=\varphi(\widehat{T}_l)e_1
			=\varphi(\widehat{T}_l)\widehat{W}^H\varphi(A)^{-1}b,\\
			\varphi(\widehat{T}_l)^{-1}\widehat{W}^Hb&=\widehat{W}^H\varphi(A)^{-1}b=e_1.\\
		\end{aligned}
	\end{equation}
	Similarly, we  obtain
	\begin{equation}\label{eqn:e1_c_left}
		\begin{aligned}
			e_1=\psi(\widehat{K}_l)^{-1}\widehat{V}^Hc,
			\qquad
			e_1^H=c^H\widehat V\psi(\widehat{K}_l)^{-H}=c^H\widehat V\psi(\widehat{T}_l)^{-1}.\\
		\end{aligned}
	\end{equation}
	The proof ends by rewriting Proposition \ref{thm:moment_mathcing_weighted}
	with \eqref{eqn:e1_b_right} and  \eqref{eqn:e1_c_left}.
	%
	%
	%
\end{proof}

For simple presentation, we assume that $\widehat m_0=1,$ where 
$\widehat m_0:=c^H\psi(A)^{-1}\varphi(A)^{-1}b.$ If  $\widehat m_0 \neq 1,$ then a similar discussion of Remark \ref{rmk:assumption_cb_equal_1} can be  applied.
In a manner analogous to the treatment in \cite{MR4311639}, the linear functional with a weighted term is defined as follows.
\begin{equation}\label{eqn:int_definition_weighted}
	\begin{aligned}
		\widehat{\mathfrak{I}}(\mathcal {F})&:=c^H\psi(A)^{-1}\mathcal {F}(A)\varphi(A)^{-1}b.\\
	\end{aligned}
\end{equation}

With the bases from (\ref{eqn:generalized_nonsymmetric_Lanczos}), we have the Gauss-Christoffel quadrature formula:
\begin{equation}\label{eqn:Gauss_quad_formula_weighted}
	\begin{aligned}
		\widehat{\mathfrak{I}}(\mathcal {F})&=\widehat{\mathfrak{I}}_G(\mathcal {F})&+&\widehat{\mathfrak{E}}(\mathcal {F}),\\
		\ie, \quad  c^H\psi(A)^{-1}\mathcal {F}(A)\varphi(A)^{-1}b&=e_1^H\mathcal {F}(\widehat{T}_l)e_1&+
		& \widehat{\mathfrak{I}}([\Lambda(\lambda)]^2\mathcal {F}[\lambda_1,\lambda_1,\lambda_2,\lambda_2,\ldots,\lambda_l,\lambda_l,\lambda]).\\
	\end{aligned}
\end{equation}

As Proposition \ref{thm:moment_mathcing_weighted} shows, the equality $\widehat{\mathfrak{I}}(\mathcal {P})=\widehat{\mathfrak{I}}_G(\mathcal {P})$ is satisfied for any $\mathcal {P}(\lambda)\in \mathbb{P}_{2l-1}.$ 
It can thus be surmised that $\widehat{\mathfrak{I}}_G(\mathcal {F})=e_1^H\mathcal {F}(\widehat{T}_l)e_1$ represents the Gauss-Christoffel quadrature formula.
%
%
%
%
%
%
After substituting $\widehat{\mathcal {H}}(\lambda,\mathfrak{s})=\varphi(\lambda)\psi(\lambda)/(\mathfrak{s}-\lambda)$ into $\mathcal {F}(\lambda)$ in (\ref{eqn:Gauss_quad_formula_weighted}), we  obtain the following quadrature remainder formula.
\begin{equation*}
	\begin{aligned}
		c^H\psi(A)^{-1}[\varphi(A)\psi(A)(\mathfrak{s}I-A)^{-1}]\varphi(A)^{-1}b&=e_1^H\psi(\widehat{T}_l)(\mathfrak{s}I-\widehat{T}_l)^{-1}\varphi(\widehat{T}_l)e_1+\widehat{\mathfrak{E}}(\widehat{\mathcal {H}}),\\
		c^H(\mathfrak{s}I-A)^{-1}b&=c^H\widehat{V}(\mathfrak{s}I-\widehat{W}^HA\widehat{V})^{-1}\widehat{W}^Hb+\widehat{\mathfrak{E}}(\widehat{\mathcal {H}}),\\
	\end{aligned}
\end{equation*}
where  $\varphi(\widehat{T}_l)e_1=\widehat{W}^Hb$ and $e_1^H\psi(\widehat{T}_l)=c^H\widehat{V}$ are  described in (\ref{eqn:e1_b_right}) and  (\ref{eqn:e1_c_left}), respectively.  
In conjunction with   \eqref{eqn:divied_difference_finite} and  \eqref{eqn:int_definition_weighted}, we can derive the following form of the MOR error.
\begin{equation}\label{eqn:error_integration_finite}
	\begin{aligned}
		\quad e(\mathfrak{s})=h(\mathfrak{s})-\tilde{h}(\mathfrak{s})=\widehat{\mathfrak{E}}(\widehat{\mathcal {H}})
		&=\widehat{\mathfrak{I}}( [\Lambda(\lambda)]^2\widehat{\mathcal {H}}[\lambda_1,\lambda_1,\lambda_2,\lambda_2,\ldots,\lambda_l,\lambda_l,\lambda])\\
		&=\left\{\frac{\varphi(\mathfrak{s})\psi(\mathfrak{s})}{[\Lambda(\mathfrak{s})]^2}\right\}c^H\psi(A)^{-1}\Lambda(A) (\mathfrak{s}I-A)^{-1}\Lambda(A) \varphi(A)^{-1}b.\\
	\end{aligned}
\end{equation}


We conclude by outlining the explanation with respect to  variable $\lambda$.
Proposition \ref{prop:Hermitian_interpolation_of_lamda} represents the Hermitian interpolation formula.
The integral is analogous to the linear functional with a weighted term \eqref{eqn:int_definition_weighted}.
The Gauss-Christoffel quadrature  formula is presented in  \eqref{eqn:Gauss_quad_formula_weighted}.
The Ritz values $\lambda_i$ concern the interpolation and quadrature nodes. 
Propostion \ref{thm:moment_mathcing_weighted} and Propostion \ref{th:algebraic_precision}
reveal that  Gauss quadrature has degree of precision $(2l-1)$.
Finally, the error formula $e(\mathfrak{s})$ corresponds to the Gauss quadrature remainder.
More details can be found in the  manuscript \cite{2verlin2021transferfunctioninterpolationremainder}.

This section of the research is based on the fact that the Ritz values represent the quadrature nodes of a  Gauss-Christoffel quadrature.
 This leads to the motivation to establish the relationship between the error formula and the moments matching. 
The book \cite{MR3024841} by Liesen and Strako{\v{s}}  provides a comprehensive overview of the relationships among  the symmetric Lanczos procedure ($A=A^H$  with $b=c$), moments matching, orthogonal polynomials, continued fractions and Gauss quadrature. 
Recently, Strako{\v{s}} and his co-workers have a series of work on generalizing it onto  the  Lanczos biorthogonalization procedure 
\cite{MR4311639,MR3671502,MR3871800,MR3868647,MR2505848}.
For further details  on  the  Gauss quadrature, we refer
the readers to \cite{Freund1993,MR3688917,MR2456794,deckers2009orthogonal,MR2582949,MR4212916,MR4433120,MR4364913} and references therein.

\subsection{The MOR error  for the descriptor model: $E\neq I$}\label{section:descriptor_model}
This section presents a discussion of the  approximation of  $h(\mathfrak{s})=c^H(\mathfrak{s}E-A)^{-1}b$
by $\tilde{h}(\mathfrak{s})=c^H\mathcal{V}(\mathfrak{s}\mathcal W^HE\mathcal V-\mathcal W^HA\mathcal V)^{-1}\mathcal W^Hb$.
%
%
For the sake of simplicity, we assume  that $E$ is invertible.
Consequently, we obtain 
$h(\mathfrak{s})=c^H[\mathfrak{s}I-(E^{-1}A)]^{-1}(E^{-1}b).$
%
Substituting it into Theorem \ref{thm:main_result_1}, we  obtain the error formula  of  MOR for the descriptor model.
\begin{equation*}
\begin{aligned}
\tilde{h}(\mathfrak{s})&=c^H\mathcal V(\mathfrak{s}\widetilde{\mathcal W}^H\mathcal V-\widetilde{\mathcal W}^HE^{-1}A\mathcal V)^{-1}\widetilde{\mathcal W}^H(E^{-1}b)\\
&=c^H\mathcal V[\mathfrak{s}(E^{-H}\widetilde{\mathcal W})^HE\mathcal V-(E^{-H}\widetilde{\mathcal W})^HA\mathcal V]^{-1}(E^{-H}\widetilde{\mathcal W})^Hb\\
&=c^H\mathcal V(\mathfrak{s}\mathcal W^HE\mathcal V-\mathcal W^HA\mathcal V)^{-1}\mathcal W^Hb,\\
\end{aligned}
\end{equation*}
where $\mathcal{V}$ and $\widetilde{\mathcal W}$ are set up by $(E^{-1}A,E^{-1}b,c)$
in Theorem \ref{thm:main_result_1}. Later, we shall use $\mathcal W=E^{-H}\widetilde{\mathcal W}$ instead of $\widetilde{\mathcal W}.$

\begin{theorem}\label{th:main_result_include_E}
	Let 
$\span(\mathcal V)=\cks(E^{-1}A,E^{-1}b,k_b,m_b)$ and 
$\span(\mathcal W)=\cks((AE^{-1})^H,E^{-H}c,k_c,m_c)$ with
$k_b+m_b=k_c+m_c=:l.$
Suppose that the Lanczos biorthogonalization procedure of $A$, $[\prod_{j=2}^{k_b}(A-s_jE)^{-1}E](A-s_1E)^{-1}b$ and $\prod_{j=2}^{k_c}(A-t_jE)^{-H}E^H](A-t_1E)^{-H}c$ does not break down until the $(l+1)$th iteration. With notations \eqref{eqn:prod_all_shifts} and \eqref{eqn:g_b_g_c_expression}, it holds that 
\begin{equation*}
\begin{aligned}
e(\mathfrak{s})=h(\mathfrak{s})-\tilde{h}(\mathfrak{s})&=c^H(\mathfrak{s}E-A)^{-1}b-c^H\mathcal V(\mathfrak{s}\mathcal W^HE\mathcal V-\mathcal W^HA\mathcal V)^{-1}\mathcal W^Hb\\
&=\frac{1}{g_b(\mathfrak{s})g_c(\mathfrak{s})}c^Hg_c(E^{-1}A)(\mathfrak{s}I-E^{-1}A)^{-1}g_b(E^{-1}A)E^{-1}b\\
&=\frac{1}{g_b(\mathfrak{s})g_c(\mathfrak{s})}c^Hg_c(E^{-1}A)g_b(E^{-1}A)(\mathfrak{s}E-A)^{-1}b,\\
\end{aligned}
\end{equation*}
%
where  $\Lambda(\lambda)$
is the monic characteristic polynomial of $(\mathcal W^HE \mathcal V)^{-1}\mathcal W^HA\mathcal V$.

\end{theorem}
\section{The error formula of  the one-sided  projection method}\label{sect:one_sided_error_formula_section}

%

It is evident that the use of known methods can be applied to approximate  $\mathcal{H}(A)b=(\mathfrak{s}I-A)^{-1}b$. 
It is the action of a matrix function on a vector,which is  discussed in \cite[Chapter 13]{MR2396439} and \cite{MR2528942,MR3095912,guttel2010rational}. 
Once the pre-multiplication by $c^H$ has been performed, an approximation of the transfer function will be obtained.
%
%
%
%
The basis is derived from the following:
\begin{equation}\label{eqn:V_2l_expression}
\begin{aligned}
	\cks(E^{-1}A,E^{-1}b,k,m):&=\RK(E^{-1}A,E^{-1}b,\mathbb{S},k)\cup \Krylov(E^{-1}A,E^{-1}b,m),\\
\span(V_{2l})=\cks(E^{-1}A,E^{-1}b,k,m)&=
\Krylov(E^{-1}A,\varphi_{k}(E^{-1}A)^{-1}E^{-1}b,2l),\\
%
V^HV=I, \qquad \varphi_{k}(\lambda)&=\prod\limits_{j=1}^{k}(\lambda-s_j),
\qquad k+m=2l.\\
\end{aligned}
\end{equation}
By summarizing \cite[(2.1)]{MR2854603} and \cite[(2.2)]{MR2858340},
we obtain the error of the one-sided projection method for MOR.
%
%
\begin{theorem}\label{thm:one_sided_error}
Suppose the dimension of $\cks(E^{-1}A,E^{-1}b,k,m)$ is $2l$.
With notations (\ref{eqn:V_2l_expression}),  it holds that
\begin{equation*}
\begin{aligned}
h(\mathfrak{s})-\breve{h}(\mathfrak{s})&=c^H(\mathfrak{s}E-A)^{-1}b-c^HV(\mathfrak{s}V^HEV-V^HAV)^{-1}V^Hb\\
&=\frac{1}{G_{\rm one}(\mathfrak{s})}c^HG_{\rm one}(E^{-1}A)(\mathfrak{s}E-A)^{-1}b,
\end{aligned}
\end{equation*}
where
\begin{equation}\label{eqn:G_one_Galerkin_form}
\begin{aligned}
G_{\rm one}(\lambda):=\frac{\bigwedge_{\rm one}(\lambda)}{\varphi(\lambda)},
\end{aligned}
\end{equation}
with $\bigwedge_{\rm one}(\lambda)=\prod_{i=1}^{2l}(\lambda-\lambda_i)$ be 
the monic characteristic polynomial of $(V^HEV)^{-1}V^HAV$.

\end{theorem}

We observe $G_{\rm one}(\lambda)$ has a similar  form to $G_{\rm two}(\lambda)$ in  (\ref{eqn:G_two_form}), which is formed  in the two-sided projection method.
The primary difference is that the order of the reduced system  obtained by the two-sided  projection is  $l$, while
the one-sided projection yields $2l$.
In comparison  to  the numerical quadrature, it is readily  comprehensible.
To attain the  degree of precision $2l-1$, it is possible to employ Gauss quadrature (two-sided  projection method, cf. Section \ref{subsect:Hermitian_remainder_lambda}), or alternatively, to utilize an interpolating polynomial of degree  $2l-1$ to perform the integral (one-sided  projection method, cf. Section~\ref{sect:one_sided_error_formula_section}).
This phenomenon has also been observed in numerical experiments (cf. Section \ref{sect:numerical_experiments}).
\section{Interpolatory $H_{\infty}$ norm MOR}
\label{sect:H_infty_norm_section}
In this section, we assume $A$ is $c$-stable, which implies that the eigenvalues of A are located  on left-hand  half plane.
The norm $\|\cdot\|_{H_\infty}$ is defined in a $\mathcal {H}_{\infty}$-space. The space requires the function matrix to be analytic and bounded in the open right-hand half plane. 
For  $h(\mathfrak{s})\in \mathcal {H}_\infty^n$, we have $\|h(\mathfrak{s})\|_{H_\infty}=\sup_{\mathfrak{s}\in i\mathbb{R}\cup\{\infty\}} |c^H(\mathfrak{s}I-A)^{-1}b|.$
The MOR in the $H_{\infty}$ norm sense is to solve the optimization problem:
\begin{eqnarray}
\|h(\mathfrak{s})-\hat{h}_*(\mathfrak{s})\|_{H_{\infty}}&=&\min\limits_{\hat{h}(\mathfrak{s})\in \mathcal {H}_\infty^l}
\|h(\mathfrak{s})-\hat{h}(\mathfrak{s})\|_{H_{\infty}},\label{eq:H_infty_problems}\\
	\mbox{where} \quad  
		h(\mathfrak{s})&=&c^H(\mathfrak{s}E-A)^{-1}b,\nonumber\\
	\hat{h}(\mathfrak{s})&=&c_l^H(\mathfrak{s}E_l-A_l)^{-1}b_l+d_l.\label{eqn:reduced_h_hat}
	\end{eqnarray}
The motivation of  the $H_{\infty}$ norm MOR is clearly described in
\cite[Section 1]{MR3068159}.
More theories and algorithms can be found in 
\cite{MR1205327,MR1363936}
and references therein.
It is worth noting that balanced truncation methods \cite{MR609248,MR465491}
 provide an  $H_{\infty}$ norm  estimation  of
the model error by using the Hankel singular values \cite{MR2503780,MR2155615}.



%

We appoximately solve \eqref{eq:H_infty_problems} by solving the interpolatory $H_{\infty}$ norm MOR.
Concretely speaking, the function  $\hat h(\mathfrak{s})$ in \eqref{eq:H_infty_problems} is constrained to be equal to $\tilde h(\mathfrak{s})$, which is an  interpolatory transfer function.
Therefore, the reduced  transfer function $\tilde h(\mathfrak{s})$ is a rational interpolating function of $h(\mathfrak{s})$. This implies that $\tilde h(\mathfrak{s})$ can be formed by a rational Krylov subspaces projection method.
For simplicity, we shall set $d_l=0$.  As the order $l$ increases, the error $\|h(\mathfrak{s})-\tilde{h}(\mathfrak{s})\|_{H_{\infty}}$ decreases.


%


%
%
%
%
%

%
%


In seeking the optimal shifts for $\tilde h(\mathfrak{s})$, it is evident that the shifts must   fully influence the error $e(\mathfrak{s})$.
Hence, we set  $k_b=k_c=l.$
%
%
In the words of finding the optimal shifts, the interpolatory $H_{\infty}$ norm MOR can be expressed as the following optimization problem.
\begin{equation}\label{eqn:min_sup_problem_H_infty}
\begin{aligned}
\min_{s_1,s_2,\ldots,s_{l},t_1,t_2,\ldots,t_{l}}\sup \limits_{\mathfrak{s}\in i\mathbb{R}\cup\{\infty\}}|e(\mathfrak{s})|
&=
\min_{s_1,s_2,\ldots,s_{l},t_1,t_2,\ldots,t_{l}}\sup \limits_{\mathfrak{s}\in i\mathbb{R}\cup\{\infty\}}|h(\mathfrak{s})-\td h(\mathfrak{s})|\\
&=\min_{s_1,s_2,\ldots,s_{l},t_1,t_2,\ldots,t_{l}}\sup \limits_{\mathfrak{s}\in i\mathbb{R}\cup\{\infty\}}
\left|\frac{1}{G(\mathfrak{s})}\right|\left|c^HG(E^{-1}A)(\mathfrak{s}E-A)^{-1}b \right|,\\
\end{aligned}
\end{equation}
where $G(\mathfrak{s})$ is either $G_{\rm two}(\mathfrak{s})$ in \eqref{eqn:G_two_form} or  $G_{\rm one}(\mathfrak{s})$ in \eqref{eqn:G_one_Galerkin_form}.
Given that $(\mathfrak{s}E-A)$ is an ${n \times n}$ matrix, the complete computation of $e(\mathfrak{s})$ is time-consuming.
Consequently, the development of effective algorithms demands the utilization of certain approximations of $e(\mathfrak{s})$.
\begin{remark}\label{rmk:ez_estimations}
We establish the following  approximations for $|e(\mathfrak{s})|$.
\begin{equation*}
	\begin{aligned}
		e(\mathfrak{s})=h(\mathfrak{s})-\tilde{h}(\mathfrak{s})
		&=\frac{1}{G(\mathfrak{s})}c^HG(E^{-1}A)(\mathfrak{s}E-A)^{-1}b,\\
		{\rm Approximation \quad 1}: \quad|e(\mathfrak{s})|&\approx \mathcal {C}_1\frac{1}{|G(\mathfrak{s})|},\\
		{\rm Approximation \quad 2}: \quad|e(\mathfrak{s})|&\leq \frac{1}{|G(\mathfrak{s})|}\|c^HG(E^{-1}A)\|_2 \|(\mathfrak{s}E-A)^{-1}b\|_2\\
		&\approx \frac{1}{|G(\mathfrak{s})|}\mathcal {C}_2\|V(\mathfrak{s}W^HEV-W^HAV)^{-1}W^Hb\|_2\\
		&= \mathcal {C}_2\frac{1}{|G(\mathfrak{s})|}\|(\mathfrak{s}W^HEV-W^HAV)^{-1}W^Hb\|_2.\\
	\end{aligned}
\end{equation*}
\end{remark}

The term $G(\mathfrak{s})$ is a quantity formula, which implies that  it can be computed easily.
The one-sided projection algorithm  in \cite{MR2858340}  is designed by  using Approximation 1
(cf. Algorithm \ref{alg:scl}).
%
Except $1/G(\mathfrak{s})$, the left of $e(\mathfrak{s})$ is still dependent on $\mathfrak{s}$. This implies that Approximation 1 can be improved.
In Approximation 2, we  employ  $\|V(\mathfrak{s}W^HEV-W^HAV)^{-1}W^Hb\|_2$ to approximate $\|(\mathfrak{s}E-A)^{-1}b\|_2$.
%
%
%
In our experimental examples (cf. Table~\ref{tbl:small_problems_CPU}),
it can be observed that   algorithms based on Approximation 2  exhibit better behaviour than those based on Approximation 1.

\section{Greedy algorithms for MOR}\label{section:greedy_algorithms_two_sided}
As outlined in the books\cite{MR2155615,MR2516498}, a range of methodologies exists for MOR.
In the field of rational Krylov subspace projection methods,
there exist numerous adaptive algorithms based on various error estimations \cite{frangos2007adaptive,MR2501562,MR3473700,MR3592472,panzer2014model}.
By using Remarking~\ref{rmk:ez_estimations}, we propose a greedy two-sided projection algorithm,
which  is  a modification of  a one-sided projection algorithm in \cite{MR2858340}.

\subsection{Algorithm \ARKSM}

\begin{algorithm}[t]
\caption{ Adaptive Rational Krylov Subspace Method (\ARKSM) for MOR\label{alg:scl}\cite{MR2858340}}
{\bf Input}: $A,E \in \mathbb{C}^{n \times n}, b,c \in \mathbb{C}^{n \times 1}.$

{\bf Input}: $s_{\min}, s_{\max}, l_{\max}$.

\begin{algorithmic}[1]
\STATE $s_1=s_{\min},v_1=(A-s_1E)^{-1} b;$
\STATE $s_2=s_{\max},v_2=(A-s_2E)^{-1} b;$
\STATE  $V=[v_1,v_2];$
 $\quad V=\orth(V);$
\FOR{$l=2, \ldots ,l_{\max}$}
    \STATE  Update $V^HAV;V^HEV;V^Hb;$
    \STATE Get Ritz values $\lambda_i \in \eig((V^HEV)^{-1}V^HAV)$;

    \STATE  Determine $\partial \Xi=$ convex hull of $\{-\lambda_1,\ldots,-\lambda_l,s_{\min},s_{\max}\}$;

     \STATE Select the next shift $s_{l+1}$:
\begin{equation}\label{eqn:ARKSM_next_shift}
\begin{aligned}
s_{l+1}
&=\arg\max\limits_{\mathfrak{s}\in\partial\Xi}\left|\frac{\prod\limits_{j=1}^{l}(\mathfrak{s}-s_j)}{\prod\limits_{j=1}^{l}(\mathfrak{s}-\lambda_j)}\right|.\\
\end{aligned}
\end{equation}

    \STATE $v_{l+1}=(A-s_{l+1}E)^{-1} b;$
   \STATE  $V=\orth([V,v_{l+1}]);$
\ENDFOR

\end{algorithmic}
{\bf Output}: $V$; Shifts $s_l (l=1,2,\ldots,l_{\max})$; $\quad \breve{h}(\mathfrak{s})=c^HV(\mathfrak{s}V^HEV-V^HAV)^{-1}V^Hb \approx h(\mathfrak{s}).$
\end{algorithm}

In paper \cite{MR2858340}, it is required that the shifts should be   either real numbers or conjugate pairs. 
For simplicity, we do not impose this constraint. 
The one-sided projection algorithm is  outlined in Algorithm~\ref{alg:scl}.
The concept of the greedy algorithm is employed in \eqref{eqn:ARKSM_next_shift} for computing the optimal shifts.
As stated in Remark~\ref{rmk:ez_estimations} Approximation 1, the error $|e(\mathfrak{s})|$ is actually approximated by $\mathcal {C}_1/|G(\mathfrak{s})|$.
The initial  two shifts are acquired by
\begin{equation}\label{eqn:definition_of_s_min_s_max}
\begin{aligned}
s_{\min}=\mbox{\texttt{eigs(-A,E,1,`sm')}},
\quad 
s_{\max}=\mbox{\texttt{eigs(-A,E,1,`lm')}}.
\end{aligned}
\end{equation}
%
%
%
%
The other shifts are selected on the boundary of $\Xi$, which is a mirror region of  the Ritz values.
The designation of $\Xi$ is based on the observation that the optimal $H_2$ norm MOR condition is 
$s_i=-\overline{\lambda_i}$ \cite{MR2421462}.
By applying the maximal value theorem to $\Xi$, one can ascertain that the shifts are located on the boundary $\partial\Xi$.
Another explanation can be found in \cite[Section 4.1]{MR3095912}.

\subsection{\textsf{Two-sided} algorithms}

\begin{algorithm}[t]
	\caption{\textsf{Two-sided} greedy	rational Krylov subspace method for MOR \label{alg:two_side_adaptive}}
	{\bf Input}: $A,E \in \mathbb{C}^{n \times n}, b,c \in \mathbb{C}^{n \times 1}.$
	
    {\bf Input}: $\alpha,\beta,k_{\rm two},s_{\max}, l_{\max}$, Option number in  \eqref{eqn:two_sides_s_select}.
	
	\begin{algorithmic}[1]
		\STATE $s_1=|s_{\max}|\iota /10,v_1=(A-s_1E)^{-1} b;$
		\STATE $t_1=-|s_{\max}|\iota /10,w_2=(A-s_2E)^{-H} c;$
		\STATE 
		$V=\orth(v_1); W=\orth(w_1);$
		
		\STATE  Determine the set for choosing shifts: 
		$\mathbb{Z}_2=\mathbb{Z}(\alpha,\beta,k_{\rm two})$ from (\ref{eqn:image_samples});
		\FOR{$l=1, \ldots ,l_{\max}$}
		\STATE  Update $W^HAV;W^HEV;W^Hb;c^HV;$
		\STATE Get Ritz values $\lambda_i \in \eig((W^HEV)^{-1}W^HAV);$
		$\quad \Lambda(\mathfrak{s}):=\prod_{i=1}^{l}(\mathfrak{s}-\lambda_i);$
		\STATE Set new symbols:
		$\varphi_l(\mathfrak{s}):=\prod_{i=1}^{l}(\mathfrak{s}-s_i);\quad \psi_l(\mathfrak{s}):=\prod_{i=1}^{l}(\mathfrak{s}-t_i);$
		
		\STATE Select next shift for $s_{l+1}$:
		{\small
			\begin{equation}
				\label{eqn:two_sides_s_select}
				\begin{aligned}
					&{\rm Option \quad 1:}& s_{l+1}&=\arg\max\limits_{\mathfrak{s}\in\mathbb{Z}_2}\left| \frac{\varphi_l(\mathfrak{s})\psi_l(\mathfrak{s})}{[\Lambda(\mathfrak{s})]^2}\right|,\\
					&{\rm Option \quad 2:}& s_{l+1}&=\arg\max\limits_{\mathfrak{s}\in\mathbb{Z}_2}\left| \frac{\varphi_l(\mathfrak{s})\psi_l(\mathfrak{s})}{[\Lambda(\mathfrak{s})]^2}\right|\|(\mathfrak{s}W^HEV-W^HAV)^{-1}W^Hb\|_2,\\
				\end{aligned}
			\end{equation}
		}
		
		
		
		\STATE Select the next shift for $t_{l+1}$:
		\begin{equation*}
		\begin{aligned}
		t_{l+1}&=\texttt{conj}(s_{l+1});\\
	\end{aligned}
	\end{equation*}
		
%
		\STATE $v_{l+1}=(A-s_{l+1}E)^{-1} b;w_{l+1}=(A-t_{l+1}E)^{-H}c;$
		\STATE  $V=\orth([V,v_{l+1}]);W=\orth([W,w_{l+1}]);$
		\ENDFOR
		
	\end{algorithmic}
	{\bf Output}: $V,W,$
	$\quad \tilde{h}(\mathfrak{s})=c^HV(\mathfrak{s}W^HEV-W^HAV)^{-1}W^Hb \approx h(\mathfrak{s}).$
\end{algorithm}
We propose the greedy \twosided~algorithm in  Algorithm~\ref{alg:two_side_adaptive}.
Since $E^{-1}A$ is c-stable, it is preferable to utilize shifts on the right half-plane in order to guarantee that $(A-s_iE)$ remains nonsingular.
%
The boundary of the right  half plane is the imaginary axis.
Moreover, we are currently doing research on $H_{\infty}$ norm  MOR, which is defined on the  imaginary  axis.
Thus, we shall select the shifts on the imaginary axis. Selecting shifts along the imaginary axis is a non-novel operation
 (e.g.\cite{frangos2007adaptive,MR2501562}).
In practice, our shifts are selected on $\mathbb{Z}(\alpha,\beta,500)$, where
\begin{equation}\label{eqn:image_samples}
	\begin{aligned}
		\mathbb{Z}(\alpha,\beta,k)=[-\iota\times\texttt{logspace}(\alpha,\beta,k),0,\iota\times\texttt{logspace}(\alpha,\beta,k)].
	\end{aligned}
\end{equation}

In the \twosided~algorithm, 
the function $G(\mathfrak{s})$ is $G_{\rm two}(\mathfrak{s})$ in
\eqref{eqn:G_two_form}.
Since the variable  $\mathfrak{s}$ in $e(\mathfrak{s})$ is actually independent of the obtained shifts and Ritz values,
the  greedy algorithm is an appropriate method for solving \eqref{eqn:min_sup_problem_H_infty}.
 The next  shifts are obtained by identifying the maximal point of the current error.
On Line 9, we provide two potential options for the next shift, which
correspond to the approximations of $e(\mathfrak{s})$ in Remark \ref{rmk:ez_estimations}.

Moreover, we claim that the obtained shifts are distinct,  given that any obtained shifts satisfy $e(s_j)=0,$ which implies that they cannot be the maximal points of the error. In other words, the shifts obtained will not be selected in the subsequent iterations.
This fact  implies that  we can
orthogonalize $[V,(A-s_iE)^{-1}b]$ instead of  $[V,(A-s_iE)^{-1}v],$ where $v$ is the last vector of $V$.
This is the process that is undertaken on Line 12.

The computation of \eqref{eqn:two_sides_s_select} is 
inexpensive. The computation of 
$\varphi_l(\mathfrak{s}),\psi_l(\mathfrak{s})$ and $\Lambda(\mathfrak{s})$ merely involves quantity calculations.
In order to implement Option 2, it is necessary to solve the following
small order  linear equations: $(\mathfrak{s}W^HEV-W^HAV)^{-1}V^Hb$.
All of the coefficient matrices are  identical, while only $\mathfrak{s}$ varies in $\mathbb{Z}_2$. Consequently,  we are able to 
  utilize \texttt{hess}$(W^HEV,W^HAV)$ and  \texttt{linsolve(..., opts.UHESS=1)}
  to reduce the CPU time.

A challenge is presented in determining the value of $t_{l+1}$ for the left subspace.
A simple decision is made by setting $t_{l+1}=\overline{s_{l+1}}$. 
A number of alternative strategies regarding $t_{l+1}$ have been investigated, which are analogous to the previously discussed strategies concerning $s_{l+1}$.
The results of our numerical testing do not indicate that those options exhibit better behaviors than the proposed strategy of   $t_{l+1}=\overline{s_{l+1}}.$

The $H_{\infty}$ norm error is quantified by the max error:
\begin{equation}\label{eqn:compute_error_ez_norm}
	\begin{aligned}
		\|e(\mathfrak{s})\|_{\infty}=\max\limits_{\mathfrak{s}\in \mathbb{Z}_e}|h(\mathfrak{s})-\tilde{h}(\mathfrak{s})|,
	\end{aligned}
\end{equation}
where $\mathbb{Z}_e$  also has form \eqref{eqn:image_samples}.
Our greedy algorithms actually are not required to compute it. 
The max error is  employed  only for the purpose of evaluation in comparison with other algorithms.
The computation of the max  error requires a significant amount of calculation.
This is not analogous to the case of matrix equations, where the residual norm can be computed in the reduced problems \cite{MR3102415,MR3299557}.

\section{Numerical experiments}\label{sect:numerical_experiments}
All  experiments are carried out in Matlab2016a  on a notebook computer (64 bits) with an Intel CPU i7-5500U and  8GB memory.
Any data involving random numbers are fixed by setting \texttt{rand(`state',0)} or \texttt{randn(`state',0)}.
It is worthy of mention that  \texttt{eigs} employs a random vector as the initial vector.
Prior to utilizing \texttt{eigs} for computing  \eqref{eqn:definition_of_s_min_s_max}, we also set \texttt{rand(`state',0)}.
We  shall compare our greedy algorithms with \ARKSM~(Algorithm \ref{alg:scl}), 
the adaptive Antoulas-Anderson (\AAA) algorithm and
the iterative rational Krylov algoriXthm  (\IRKA).

The \AAA~algorithm can be applied for MOR\cite{10.1093/imamci/3.2-3.61,MR3840899,MR3805855}. 
 As it is designed to solve a minimum-maximum problem, it can therefore be considered to approximately solve the $H_{\infty}$ norm MOR problem \eqref{eq:H_infty_problems}.
The algorithm acquires samples along the imaginary axis \cite[Section 6.9]{MR3805855}. The greedy idea is also employed, resulting in the selected samples being nested.
%
In the testing procedure, the samples set $\mathbb{Z}_A$ are constructed to have the same form of  $\mathbb{Z}(\alpha,\beta,k).$
%
When   $\mathbb{Z}_A=\mathbb{Z}_e$ is set,  the max error is directly outputted  by \AAA. 
This is an attractive advantage of  the \AAA~algorithm.

%


  The $H_2$ norm optimal MOR  is accomplished by \IRKA\cite{MR2421462,MR4180031}.
  When  \IRKA~converges, the optimal shifts are obtained, which satisfy the condition  that $t_i=s_i=-\overline{\lambda_i}.$
  The same shifts are used for spanning the left and right rational Krylov subspaces.
  The projection process reveals that $H_2$  norm optimal MOR is also a two-sided projection method.
  Therefore, its error $e(\mathfrak{s})$ can be expressed by Theorem~\ref{thm:main_result_1}.
  The two main disadvantages of \IRKA~are as follows: The convergence rate is relatively slow, and the optimal shifts are not nested
  with respect to order $l$. The advantage is that it has an optimization property in the $H_2$ norm sense. It is not necessary for \IRKA~to compute the model error in either the $H_2$ or the $H_\infty$ norm, as this is automatically optimized in the $H_2$ norm sense. 
  Nonetheless, we continue to employ its max error \eqref{eqn:compute_error_ez_norm}  as an  applicable standard for comparison with other algorithms.
  The initial shifts of \IRKA~are derived from the output shifts  of  \ARKSM.
  The \IRKA~iteration is terminated if the iteration number exceeds 100 or if the condition $\|S_j-S_{j-1}\|_2<10^{-6}\|S_{j}\|_2$ is satisfied, where $S_{j}$ represents the sorted shifts vector at the $j$th iteration.


For large-scale problems, the computation of all algorithms is primarily focused on the process of solving linear equations.
In all tests, the Matlab $\mathtt{backslash}$ function is directly employed to complete the calculations of
both $(A-s_iE)^{-1} b$ and $(A-s_iE)^{-H} c$.
In the case of an $l$-order MOR, the number of linear solvers required for each algorithm is taken into account. 
The AAA algorithm  requires the sampling of $h(\mathfrak{s})$ on the imaginary axis. 
The function $h(\mathfrak{s})$ is evaluated for $\mathfrak{s}\in \mathbb{Z}_A=\mathbb{Z}(\alpha_A,\beta_A,k_A).$
 Therefore, it requires a total of  $(2k_A+1)$ linear solvers,  which is independent of the order $l$.
 The \IRKA~algorithm requires $2j_{\max} l$ linear solvers, where $j_{\max}$ represents the final iteration number at which the algorithm terminates.
 The \ARKSM~algorithm requires $l$ linear solvers, whereas our  \twosided~algorithm requires $2l$ linear solvers.
 Note that \ARKSM~is a one-sided type of  projection algorithm.
 By contrast, both $H_2$(\IRKA) and our greedy \twosided~algorithm are two-sided types.



\textbf{Example 1 (small problems):}
The testing examples are taken from the SLICOT benchmark problems~\cite{MR1703544,MR2482336}. 
In instances wherein the model problems possess both multiple inputs and outputs, we set $\texttt{b=B(:,1)}$ and $\texttt{c=C(1,:)'}$. 
We set $\mathbb{Z}_e=\mathbb{Z}_A=\mathbb{Z}(-3,5,700)$  for computing the max error and undertaking  sampling in  \textsf{AAA}.
In \twosided~algorithms., the shifts are selected in $\mathbb{Z}_2=\mathbb{Z}(-3,5,500)$.

\begin{table}[htb]
	\caption{{\rm 40-order MOR of small problems} \label{tbl:small_problems_CPU}}
	\centering
	\begin{tabular}{|l|ccccccc|}
		\hline
		&	Matrix	&	\textsc{beam}	&	\textsc{CDplayer}	&	\textsc{eady}		&\textsc{fom}	&	\textsc{iss}	&	\textsc{random}	\\
		\hline
		&	Size	&${\rm	348 	}$&${\rm	120 	}$&${\rm	598 	}$&${\rm	1006 	}$&${\rm	270 	}$&${\rm	200 	}$\\
		
		\hline
		\multirow{6}{*}{CPU(s)}
		&	 Sampling	&${\rm	7.08 	}$&${\rm	0.20 	}$&${\rm	40.13 	}$&${\rm	0.51 	}$&${\rm	0.40 	}$&${\rm	3.76 	}$\\
		&	\AAA	&${\rm	7.67 	}$&${\rm	0.72 	}$&${\rm	40.66 	}$&${\rm	0.80 	}$&${\rm	1.01 	}$&${\rm	4.11 	}$\\
		
		&	\ARKSM	&${\rm	2.96 	}$&${\rm	2.62 	}$&${\rm	3.87 	}$&${\rm	2.67 	}$&${\rm	2.64 	}$&${\rm	2.91 	}$\\
		&	\twosided(O1)	&${\rm	3.21 	}$&${\rm	2.72 	}$&${\rm	5.61 	}$&${\rm	2.92 	}$&${\rm	2.84 	}$&${\rm	3.04 	}$\\
		&	\twosided(O2)	&${\rm	5.02 	}$&${\rm	4.72 	}$&${\rm	7.03 	}$&${\rm	4.61 	}$&${\rm	4.60 	}$&${\rm	4.79 	}$\\
		&	$H_2$(\IRKA)	&${\rm	44.85 	}$&${\rm	0.78 	}$&${\rm	235.87 	}$&${\rm	4.18 	}$&${\rm	2.02 	}$&${\rm	23.60 	}$\\
		
		\hline
		
		\multirow{5}{*}{$\|e(\mathfrak{s})\|_{\infty}$ }
		&	\AAA	&${\rm	5.83E-02	}$&${\rm	2.29E-03	}$&${\rm	2.28E-06	}$&${\rm	5.96E-12	}$&${\rm	4.54E-06	}$&${\rm	1.21E-09	}$\\
		&	\ARKSM	&${\rm	1.29E+01	}$&${\rm	3.87E+01	}$&${\rm	9.88E-04	}$&${\rm	5.01E-12	}$&${\rm	4.05E-03	}$&${\rm	1.61E-06	}$\\
		&	\twosided(O1)	&${\rm	1.12E+00	}$&${\rm	1.02E-01	}$&${\rm	1.58E-05	}$&${\rm	1.02E-11	}$&${\rm	4.26E-05	}$&${\rm	6.67E-09	}$\\
		&	\twosided(O2)	&${\rm	6.76E-01	}$&${\rm	9.25E-02	}$&${\rm	7.51E-06	}$&${\rm	3.72E-12	}$&${\rm	2.68E-05	}$&${\rm	5.25E-09	}$\\
		&	$H_2$(\IRKA)	&${\rm	4.24E-02	}$&${\rm	2.49E-02	}$&${\rm	9.77E-07	}$&${\rm	4.32E-12	}$&${\rm	1.57E-05	}$&${\rm	3.18E-09	}$\\

		\hline
		&	\IRKA\#iter	&${\rm	100 	}$&${\rm	39 	}$&${\rm	100 	}$&${\rm	100 	}$&${\rm	62 	}$&${\rm	100 	}$\\
		\hline
		
	\end{tabular}
	
	\begin{flushleft}
		\small{ 
			The ``sampling'' stage concerns the CPU times associated with the evaluation of $h(\mathfrak{s})$ for 
			$\mathfrak{s}\in \mathbb{Z}_A=\mathbb{Z}_e=\mathbb{Z}(-3,5,700).$
			As illustrated in Fig. \ref{fig:MORs_of_small_problems}, the data of \AAA~is only concerned with
			$l=29$ for \textsc{fom} and $l=31$ for \textsc{random}.
			In this table, the $H_2$ norm MOR(\IRKA) is not computed for $l < 40$.
			The initial shifts of \IRKA~are obtained from the output shifts of the \ARKSM~algorithm.
		The term ``\IRKA\#iter" denotes the iteration number of \IRKA~at the point of termination. }
	\end{flushleft}
\end{table}

\begin{figure}[htb]
	\caption{Behaviors of  different MOR for small problems
		\label{fig:MORs_of_small_problems}}
	\subfigure[$\textsc{eady}$]{
		\includegraphics[width=3.2in]{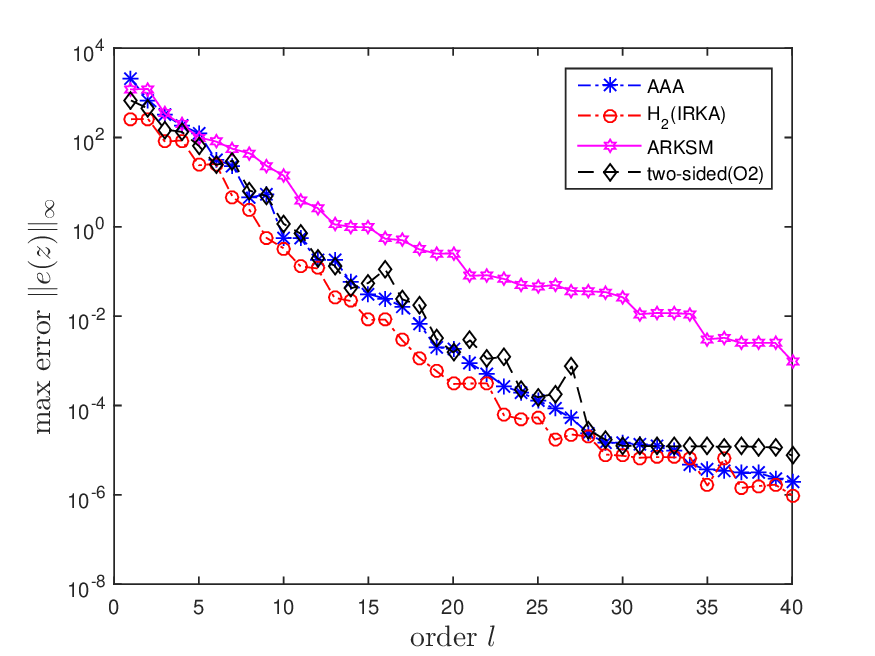}}
	\subfigure[ $\textsc{random}$]{
		\includegraphics[width=3.2in]{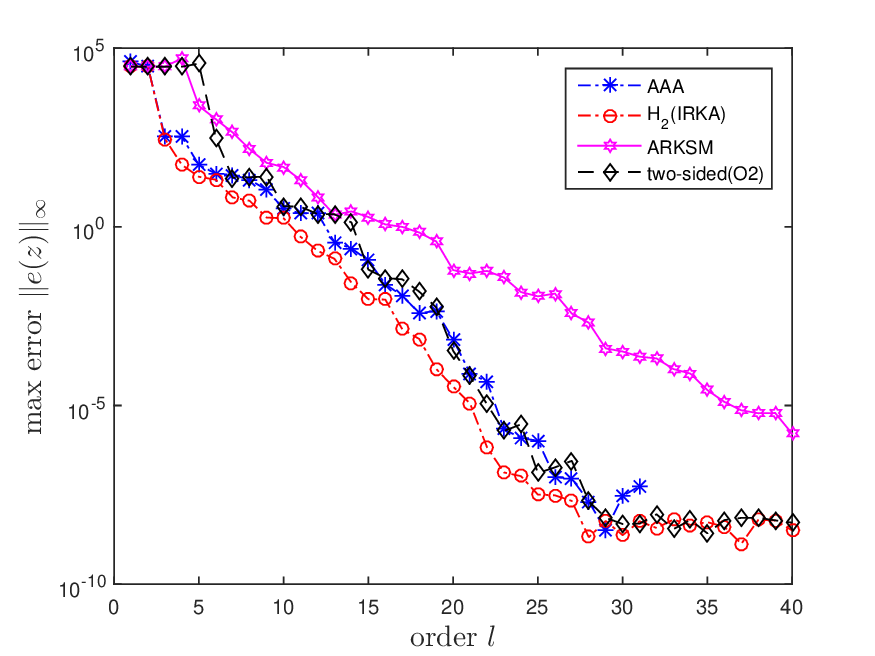}}\\
		\begin{center}
		Despite inputting $l=40$ into the \AAA algorithm, the process terminates at $l=31$ for \textsc{random}.
		\end{center}
	
\end{figure}

Two error images are presented in Fig.~\ref{fig:MORs_of_small_problems}, while the CPU times for 40-order MOR are listed in Table~\ref{tbl:small_problems_CPU}.
With the exception of the $H_2$ norm optimal MOR  (\IRKA), the remaining algorithms employ nested shifts.
Hence, their Matlab codes for producing Fig.~\ref{fig:MORs_of_small_problems} and  Table~\ref{tbl:small_problems_CPU} exhibit minimal discrepancies.
However, only the 40-order $H_2$ norm optimal MOR is computed in Table~\ref{tbl:small_problems_CPU}, while the $H_2$ norm optimal MOR of all orders ($l=1,2,\ldots,40$) is computed in Figure~\ref{fig:MORs_of_small_problems}.
 


\textbf{Example 2 (large-scale problems):}
With $E=I$, the problems \textsc{L10000} and  \textsc{L10648} respectively originate from 
papers  \cite{MR2858340} and \cite{MR2318706}.
They are derived  from the explicit discretisation  of partial differential equations.
The remaining problems originate from the Oberwolfach collection~\cite{10.1007/3-540-27909-1_11}.
A summary of the remaining information can be found in the first section of Table~\ref{tbl:large_problems_CPU}.

We set $\mathbb{Z}_A=\mathbb{Z}_e=\mathbb{Z}(\alpha,\beta,400)$ and 
$\mathbb{Z}_2=\mathbb{Z}(\alpha,\beta,500)$.
Two error images are presented in Fig.~\ref{fig:MORs_of_large_problems}, while the CPU times for 40 order MOR are listed in Table~ \ref{tbl:large_problems_CPU}.
%
%
%
The remaining settings are identical to those previously described for small problems.

\begin{figure}[htb]
  \caption{Behaviors of different MOR for large-scale problems
	\label{fig:MORs_of_large_problems}}
\centering
      \subfigure[$\textsc{flow\_v0}$]{
    \includegraphics[width=3.2in]{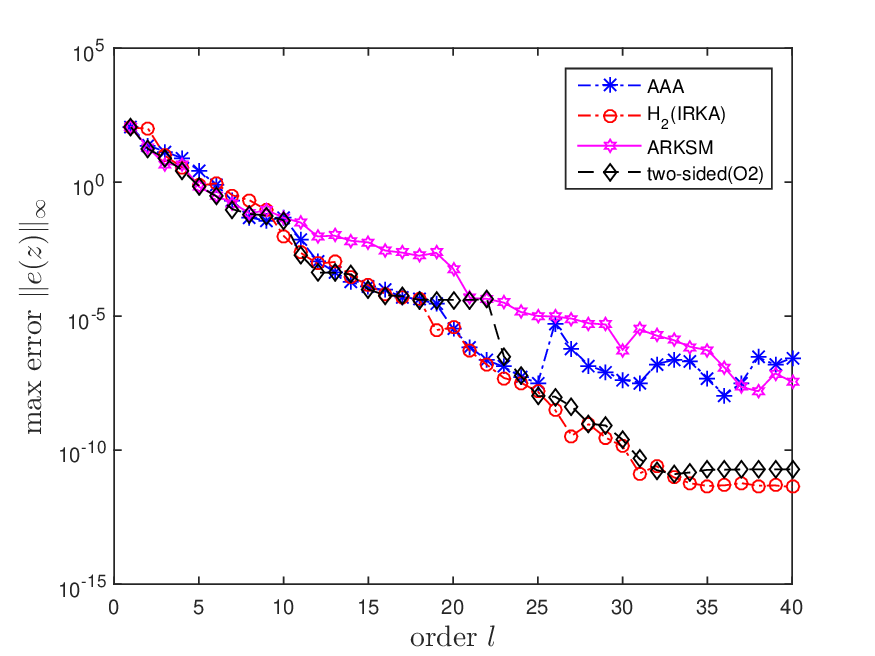} }
           \subfigure[$\textsc{rail20209}$]{
    \includegraphics[width=3.2in]{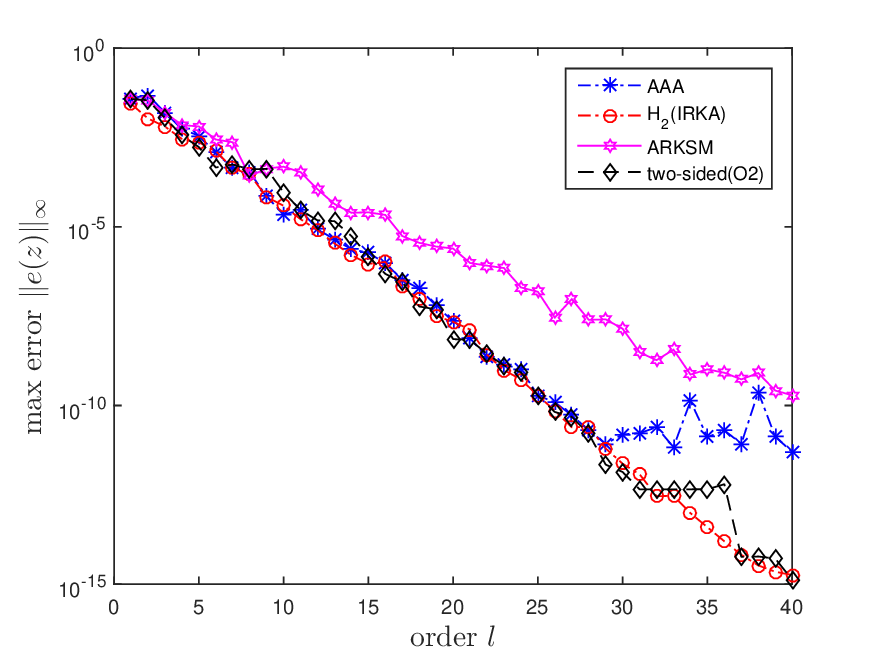} }\\

\end{figure}

\begin{table}[]
	\centering
	\caption{{\rm   40-order  MOR of large scale  problems} \label{tbl:large_problems_CPU}}
	\begin{tabular}{|c|ccccccc|}
		\hline
		&	Matrix		&	\textsc{L10000}	&	\textsc{L10648}	&	\textsc{flow\_v0}	&	\textsc{flow\_v0.5}		&	\textsc{rail5177}	&	\textsc{rail20209}	\\
		\hline
		\multirow{5}{*}{Info.}&	Size	&${\rm	10000	}$&${\rm	10648	}$&${\rm	9669	}$&${\rm	9669	}$&${\rm	5177	}$&${\rm	20209	}$\\
		&	symmetry	&${\rm	No	}$&${\rm	No	}$&${\rm	Yes	}$&${\rm	No	}$&${\rm	Yes	}$&${\rm	Yes	}$\\
		&	$b$	&${\rm	\texttt{ones(n,1)}	}$&${\rm	\texttt{ones(n,1)}	}$&${\rm	\texttt{B(:,1)}	}$&${\rm	\texttt{B(:,1)}	}$&${\rm	\texttt{B(:,6)}	}$&${\rm	\texttt{B(:,6)}	}$\\
		&	$c$	&${\rm	\texttt{randn(n,1)}	}$&${\rm	\texttt{randn(n,1)}	}$&${\rm	\texttt{C(4,:)'}	}$&${\rm	\texttt{C(1,:)'}	}$&${\rm	\texttt{C(2,:)'}	}$&${\rm	\texttt{C(2,:)'}	}$\\
		&	$\alpha,\beta$	&${\rm	-2,5	}$&${\rm	-3,4	}$&${\rm	-1,8	}$&${\rm	-3,5	}$&${\rm	-6,3	}$&${\rm	-6,3	}$\\
		
		\hline
		
		\multirow{6}{*}{CPU(s)}
		&	Sampling	&${\rm	57.94 	}$&${\rm	323.14 	}$&${\rm	70.04 	}$&${\rm	68.99 	}$&${\rm	24.10 	}$&${\rm	113.42 	}$\\
		
		&	\AAA	&${\rm	58.24 	}$&${\rm	323.47 	}$&${\rm	70.37 	}$&${\rm	69.34 	}$&${\rm	24.48 	}$&${\rm	113.81 	}$\\
		
		&	\ARKSM	&${\rm	6.21 	}$&${\rm	20.25 	}$&${\rm	4.49 	}$&${\rm	6.15 	}$&${\rm	2.71 	}$&${\rm	6.81 	}$\\
		&	\twosided(O1)	&${\rm	9.81 	}$&${\rm	41.37 	}$&${\rm	11.68 	}$&${\rm	15.30 	}$&${\rm	5.31 	}$&${\rm	19.04 	}$\\
		&	\twosided(O2)	&${\rm	12.16 	}$&${\rm	42.38 	}$&${\rm	13.31 	}$&${\rm	29.22 	}$&${\rm	7.08 	}$&${\rm	21.18 	}$\\
		&	$H_2$ (\IRKA)	&${\rm	120.56 	}$&${\rm	1073.56 	}$&${\rm	671.25 	}$&${\rm	682.88 	}$&${\rm	254.41 	}$&${\rm	1106.58 	}$\\

		\hline

		\multirow{5}{*}{$\|e(\mathfrak{s})\|_{\infty}$ }
		&	\AAA	&${\rm	1.6E-06	}$&${\rm	2.2E-06	}$&${\rm	8.3E-08	}$&${\rm	7.9E-09	}$&${\rm	2.5E-12	}$&${\rm	6.5E-11	}$\\
		&	\ARKSM	&${\rm	2.2E-02	}$&${\rm	8.9E-04	}$&${\rm	3.6E-08	}$&${\rm	1.5E-05	}$&${\rm	4.6E-10	}$&${\rm	1.9E-10	}$\\
		&	\twosided(O1)	&${\rm	8.0E-06	}$&${\rm	3.8E-06	}$&${\rm	1.3E-08	}$&${\rm	5.9E-12	}$&${\rm	6.8E-16	}$&${\rm	1.3E-15	}$\\
		&	\twosided(O2)	&${\rm	4.3E-06	}$&${\rm	1.6E-05	}$&${\rm	1.9E-11	}$&${\rm	2.5E-12	}$&${\rm	6.4E-16	}$&${\rm	1.3E-15	}$\\
		&	$H_2$ (\IRKA)	&${\rm	1.5E-07	}$&${\rm	9.7E-07	}$&${\rm	4.2E-12	}$&${\rm	3.0E-10	}$&${\rm	8.6E-16	}$&${\rm	1.8E-15	}$\\
		
		\hline
		&	\IRKA\#iter	&${\rm	21	}$&${\rm	32	}$&${\rm	100	}$&${\rm	100	}$&${\rm	100	}$&${\rm	100	}$\\

		\hline
		
		
	\end{tabular}
	
	\begin{flushleft}
		\small{ The vectors $b$ and $c$ of  the \textsc{rail20209} model are presented  in \cite[Section 5.3]{MR2421462}.
			In accordance with \eqref{eqn:definition_of_s_min_s_max},
			$10^\alpha$ identified as a lower bound of  $|s_{\min}|$, and thus, $10^\beta$ is established as an upper bound of $|s_{\max}|$.
	     	The ``sampling'' stage concerns the CPU times associated with the evaluation of $h(\mathfrak{s})$ for  $\mathfrak{s} \in \mathbb{Z}_A=\mathbb{Z}_e=\mathbb{Z}(\alpha,\beta,400).$
	     	We set $\mathbb{Z}_2=\mathbb{Z}(\alpha,\beta,500)$ for the \twosided~algorithm.
	     	The remaining settings are identical to those previously described in Table~\ref{tbl:small_problems_CPU}.
		 }
	\end{flushleft}
\end{table}

The following observations are presented for consideration.
\begin{enumerate}
	\item 
	The data presented in Tables \ref{tbl:small_problems_CPU} and \ref{tbl:large_problems_CPU} clearly demonstrate that the CPU times associated with the sampling process constitute a significant proportion of the overall processing time for the \AAA~algorithm.
	As the sampling process is not a necessary step in the execution of other algorithms, the CPU times  involved in sampling are not incorporated into their respective computations.
	Nevertheless, we continue to perform sampling for computing the max error 
	\eqref{eqn:compute_error_ez_norm}, 
	so that all of the algorithms can be compared according to the same criteria.
%
	\item 
	As illustrated in Fig.\ref{fig:MORs_of_large_problems}, the error images produced by the \twosided(O2)~algorithm are comparable to those generated by the $H_2$ norm MOR and \AAA.
   Given that the $H_2$ norm MOR exhibits $H_2$ norm optimality and that the effectiveness of \AAA~is already demonstrated in numerous examples \cite{MR3805855}, it can be stated that our greedy algorithms developed for solving the $H_{\infty}$ norm MOR demonstrate favourable behavior.
	Meanwhile, our \twosided(O2)~algorithm  requires a much shorter processing time than \IRKA~and \AAA~ in dealing with large-scale problems.

	\item Both the $H_2$ norm optimal MOR (\IRKA) and our greedy \twosided~algorithm are two-sided type projection methods.
	A total of 80 shifts have been allocated for a 40-order MOR.
	In comparison, the \ARKSM~algorithm is, in fact, a one-sided type projection method.
	For a 40-order MOR, only 40 shifts are permitted in \ARKSM.
%
	In Section~\ref{sect:one_sided_error_formula_section}, it is revealed that the forms of the errors associated with the $l$-order two-sided method and the $2l$-order one-sided method are strikingly analogous.
	It is thus proposed that the 20-order two-sided algorithm will demonstrate a similar level of precision to that exhibited by the 40-order one-sided algorithm.
		The error pictures in Fig.\ref{fig:MORs_of_small_problems} and Fig.\ref{fig:MORs_of_large_problems} illustrate this phenomenon to some extent, as evidenced by the  \textsc{eady}, \textsc{random} and \textsc{rail20209} examples.
Concretely speaking, the max error of $l=40$ in \ARKSM ~is approximately equivalent to that of $l=20$ in $H_2(\IRKA)$ or the \twosided~algorithm.
\end{enumerate}

\section{Conclusion}\label{sect:conclusion_section}
The following  is a summary of the contributions.
\begin{enumerate}
	\item  We derive an explicit error formula for MOR, when MOR is accomplished  by rational Krylov subspace methods.
	The moments matching results of MOR are particular cases of our error formula.
	\item  In the field of numerical analysis, the error formula is a valuable tool for theoretical analysis.
	We obtain  two explanations of $e(\mathfrak{s})$, regarding the variables $\mathfrak{s}$ and $\lambda$ in the resolvent function $\mathcal {H}(\lambda,\mathfrak{s})=1/(\mathfrak{s}-\lambda)$. With respect to  $\mathfrak{s}$, we observe that  $h(\mathfrak{s})$ represents  the  interpolated function,  $\tilde{h}(\mathfrak{s})$ represents the rational  interpolating function, and the model error $e(\mathfrak{s})$  is given by the Hermitian interpolation remainder.
	\item  In the situation  regarding  $\lambda$, the function $\widehat{\mathcal{H}}(\lambda)=\varphi(\lambda)\psi(\lambda)/(\mathfrak{s}-\lambda)$ has a non-relevant parameter $\mathfrak{s}$. The integral is now expressed by the linear functional with a weighted term \eqref{eqn:int_definition_weighted}. We demonstrate that 
	$h(\mathfrak{s})$ represents   the integral of $\widehat{\mathcal{H}}(\lambda)$,
   that $\tilde{h}(\mathfrak{s})$ represents the Gauss-Christoffel quadrature formula, and that  the model error $e(\mathfrak{s})$ is, in fact,  the  quadrature remainder.
     \item A thorough examination of the form of the errors reveals a similarity between the errors in $2l$-order one-sided projection methods and those in $l$-order two-sided projection methods. The numerical experiments also demonstrate this similarity. 
   
   \item   We obtain an  affortable approximation of $e(\mathfrak{s})$, which can be computed  in the reduced problem.
   By using  the approximation, we propose a greedy two-sided projection method for the interpolatory $H_{\infty}$ norm MOR.
   The efficiency of the algorithm is evidenced by numerical experiments.
\end{enumerate}
We explain two documented  phenomena.
1. The error formula  is independent of the stability of  the system. 
Consequently, the projection methods are incapable of preserving the stability unless additional constraints are introduced.
2. The final error  formula is independent of the bases  in rational Krylov subspaces.
For numerical stability, the researchers have elected to utilize the orthonormal  bases in both   left and right subspaces.

%
	
%
A challenge arises in attempting to generalize the error formula to a multi-input-multi-output system.
Once the concise and explicit expressions of the residuals have  been obtained,  the error formula of
the full interpolation \cite[Section 3.3]{doi:10.1137/1.9781611976083}
can be obtained by $e(\mathfrak{s})=R_C^H(\mathfrak{s}I-A)^{-1}R_B$.
While \cite[Theorem 4.3]{doi:10.1137/19M1255847} describes certain properties of the residuals,
there is still a need for a more concise expression.
%
%
 Another technical challenge is the utilization of tangential interpolation.
  Once the left and right tangential directions have been fixed by setting $c=c_i$ and $b=b_i$ (cf. \cite[Theorem 3.3.1, Theorem 3.3.2]{doi:10.1137/1.9781611976083}), the error formula can be readily rewritten, involving only $A$, $Bb$ and $Cc$.
 The issue of how to express the error formula concisely when considering different left and right tangential directions remains unresolved.

\section*{Acknowledge}
The author deeply appreciates Valeria Simoncini, Ren-cang Li and Shengxin Zhu for their insightful comments.



\bibliography{math20241212}


\end{document}